\documentclass[12pt,reqno]{amsart}

\usepackage{amscd}
\input xy
\xyoption{all}
\headsep=1cm \numberwithin{equation}{section}
 \newtheorem{thm}{Theorem}[section]
 \newtheorem{cor}[thm]{Corollary}
 \newtheorem{lem}[thm]{Lemma}
 \newtheorem{prop}[thm]{Proposition}
 \theoremstyle{definition}
 \newtheorem{defn}[thm]{Definition}
 \newtheorem{example}[thm]{Example}
 \theoremstyle{remark}
 \newtheorem{rem}[thm]{Remark}

 \numberwithin{equation}{section}
 
\DeclareMathOperator{\Hom}{Hom} 
\DeclareMathOperator{\Image}{Im}
\DeclareMathOperator{\coker}{Coker}
\DeclareMathOperator{\Ker}{Ker}   
\DeclareMathOperator{\Supp}{Supp} \DeclareMathOperator{\Ass}{Ass}
 \DeclareMathOperator{\beg}{indeg}
\DeclareMathOperator{\Deg}{deg} \DeclareMathOperator{\MAX}{max}

 \DeclareMathOperator{\Ext}{Ext}\DeclareMathOperator{\Dim}{dim}
\DeclareMathOperator{\depth}{depth}
 \DeclareMathOperator{\Ht}{ht}
\DeclareMathOperator{\Reg}{reg} \DeclareMathOperator{\Tot}{Tot}
\DeclareMathOperator{\grade}{grade}
\DeclareMathOperator{\End}{end} \DeclareMathOperator{\END}{end}
 \DeclareMathOperator{\A}{\alpha}
 \DeclareMathOperator{\E}{\textit{E}}
 \DeclareMathOperator{\D}{d}
 \DeclareMathOperator{\SD}{SD}
\DeclareMathOperator{\SDC}{SDC} \DeclareMathOperator{\Var}{V}

\newcommand{\fm}{\mathfrak{m}}

\newcommand{\fp}{\frak{p}}
\newcommand{\fq}{\frak{q}}
\newcommand{\fa}{\frak{a}}
\newcommand{\fb}{\frak{b}}

\newcommand{\ra}{\rightarrow}

\begin{document}

\title[CM residuals and their Castelnuovo-Mumford regularity ]
 { Cohen-Macaulay residual intersections and their Castelnuovo-Mumford Regularity}

\author[S. H. Hassanzadeh]{S. H. Hassanzadeh}
\address{Institut de mathematiques, Universit\'e Pierre et Marie Curie,175
rue du Chevaleret, Bureau 7C8, 75013, Paris, France, and Faculty of
Mathematical Sciences and Computer, Tarbiat Moallem University, 599
Taleghani Ave., Tehran 15618, Iran.}

\email{h\_hassanzadeh@tmu.ac.ir}


\date{May 29, 2009 }

\begin{abstract} In this article we study the structure of residual intersections via constructing a finite complex  which is acyclic under some sliding depth conditions on the cycles of the Koszul complex. This complex provides information on  an ideal which coincides with the residual intersection in the case of geometric residual intersection; and is closely related to it in general. A new success obtained through  studying such a complex is to prove the Cohen-Macaulayness of residual intersections of a wide class of ideals. For example we show that, in a Cohen-Macaulay local ring, any geometric residual intersection of an ideal that satisfies the sliding depth condition is Cohen-Macaulay; this is an affirmative answer to one of the main open questions in the theory of residual intersection, \cite[Question 5.7]{HU}.

The complex we construct also provides a bound for the Castelnuovo-Mumford regularity of a residual intersection  in term of the degrees of the minimal generators. More precisely, in a positively graded Cohen-Macaulay *local ring $ R=\bigoplus _{n \geq 0} R_{n}$, if $J=\fa :I$ is a (geometric) $s$-residual intersection of the ideal $I$ such that $\Ht (I)=g>0$ and  satisfies a sliding depth condition, then  $\Reg(R/J) \leq \Reg( R) + \dim(R_0)+ \sigma( \fa) -(s-g+1)\beg(I/\fa)-s$, where $\sigma( \fa)$ is the sum of the degrees of elements of a minimal generating set of $\fa$. It is also
shown that the equality holds whenever $I$ is a perfect ideal of height 2, and the base ring $R_0$ is a field.
\end{abstract}


\maketitle


\section{Introduction}
 The notion of residual intersection was originally introduced by Artin and Nagata
\cite{AN}; it has been extensively studied by Huneke, Ulrich and others. Throughout the
paper $R$ is a  Noetherian (graded) ring. Let $I$ be an (graded) ideal of
height $g$ in the local (*local) $R$, and let $s \geq g$ be an integer; an
$s$-residual intersection of $I$ is an ideal $J$ such that $J=\fa
:I$ for some (graded) ideal $\fa \subseteq I$ with $htJ \geq s
\geq \mu(\fa)$ ($\mu$ denoting minimal number of generators).
In the case where $R$ is Gorenstein and $I$ is unmixed, this notion generalizes the concept of linkage where in that
sense the ideals $I$ and $J$ have the same height.  Two important examples of residual
intersections which also demonstrate the ubiquity of such
ideals are as follows  (these examples are given in
\cite[4.1-4.3]{H}): The ideal defined by the maximal minors of a
generic $s$ by $r$ matrix with $r<s$ is an $(s-r+1)$-residual
intersection of the ideal defined by the maximal minors of a
generic $s\times(s+1)$ matrix, which is  a perfect ideal
of height 2. As another example, suppose that $R$ is CM and $I$
is an ideal of positive height that satisfies $G_\infty$, then
the defining ideal of the extended symmetric algebra of I is a
residual intersection. We refer the reader to \cite{H} for more
information.

 The Cohen-Macaulay (from now on, abbreviated by CM) property and the canonical
module of residual intersections was carefully investigated in several works such as \cite{C}, \cite{H}, \cite{HU}, \cite{U}, .... Most of these works deeply
applied a crucial lemma of Artin and Nagata \cite[lemma 2.3]{AN},
this lemma provides an inductive argument to reduce the problem
in residual intersection to that in linkage. One of
the most important condition required for this lemma, or similar
results,  is the $G_s$ condition which bounds the local number
of generators of an ideal; precisely, we say that an
ideal $I$ satisfies $G_s$ condition, if $\mu(I_{\fp})\leq \Ht (\fp)$ for all prime
ideal $\fp$ containing $I$ such that $\Ht (\fp) \leq s-1$; $I$
satisfies $G_{\infty}$, if $I$ satisfies $G_s$ for all $s$. Other conditions which are required to obtain the mentioned properties are
some depth conditions on Koszul homology modules of $I$, such as
strongly Cohen-Macaulay (SCM) and Sliding depth condition, (SD). An explicit
resolution for the residual intersection is only known in special cases. It involves generalized Koszul
complexes and approximation complexes; see for example \cite{BKM}and \cite{KU}.

The interplay between  residual intersections and some arithmetic
subjects in commutative algebra, such as \cite{HH},  \cite{PU}, \cite{v1}, ...,
 is at the origin of  a lot of attempts to weaken the conditions which imply
  some arithmetic properties of residual intersections such as Cohen-Macaulayness.
 In spite of considerable progress in this way, the main challenge in
 the theory of residual intersection is to remove the $G_s$ condition.
 As C.Huneke and B. Ulrich mentioned in their paper \cite[Question 5.7]{HU},
 the main  open question is the following:
 \begin{quote}
 \em Suppose that $R$ is a local CM ring and $I$ is
 an ideal of $R$ which is SCM (or even has sliding depth). Let $J$  be any residual intersection of $I$. Then is $R/J$  CM ?
 \end{quote}

One of the main purposes of this paper is to answer this question, affirmatively.
The idea is that we construct a finite complex $\mathcal{C}_\bullet$ whose tail consists of free modules and
 whose beginning terms are finite direct sums of  cycles of
the Koszul complex. It is shown in Proposition \ref{p16} that this complex is acyclic
under some sliding depth conditions on  cycles of the Koszul
complex. This condition is precisely defined in \ref{d13} with the abbreviated form SDC$_k$
 for some integer $k$. We then provide some s conditions  which imply the $\SDC_k$
 condition. On the way, in Proposition \ref{p15}, we completely determine the local cohomology modules (and consequently  clarify the depth) of the last  cycle of the Koszul complex which does not coincide with the boarder. This result fairly improve a proposition of Herzog, Vasconcelos and Villareal \cite[1.1]{HVV}. This investigation ensures that in the cases where the residual intersection is close to the linkage, namely when $s-g\leq2$, the complex $\mathcal{C}_\bullet$ is acyclic without any assumption on $I$; see  Corollary \ref{c17}. Some importance of $s$-residual intersections which are close to  linkage is due to the fact that these ideals contains a class of ideals whose  Rees algebra is CM; see for example \cite{HH} and \cite{U}. The ideal which is resolved by $\mathcal{C}_\bullet$, say $K$, is quite close to the residual
intersection; indeed in Theorem \ref{t19} it is shown that $K$ is always contained in $J$ and has the same radical as $J$. Moreover,
if $I$ satisfies the  sliding the depth condition ${\rm SDC}_1$ then $K$ is CM.
Therefore, the affirmative answer to the above mentioned question is in the case where $K=J$. It is shown in Theorem \ref{t19}(iv) that if $I/\fa$ is generated by at most one element locally in height $s$ then $K=J$. In particular, if the residual is geometric, see Corollary \ref{c111}.

Having  an approximation complex for the residual intersection in hand,
 we establish a bound for the Castelnuovo-Mumford regularity of
residual intersections in terms of the degrees of their
defining ideals. Determining this bound needs several careful
studies of the degrees and the maps of $\mathcal{C}_\bullet$.
More precisely, it is shown in Theorem \ref{t22} that, in a positively graded
Cohen-Macaulay *local ring, $ R=\bigoplus _{n \geq 0} R_{n}$ which admits a graded canonical module, if $J=\fa :I$
is an  $s$-residual intersection of the ideal $I$ such that $\Ht (I)=g>0$, $I/\fa$ is generated by at most one element locally in height $s$ and  that $I$ satisfies the $\SD_1$ condition, then
\begin{center}
$\Reg(R/J) \leq \Reg( R) + \dim(R_0)+ \sigma( \fa) -(s-g+1)\beg(I/\fa)-s$.
\end{center}
This formula generalizes the previous known facts about the regularity of linked ideals.
In the course of the proof of Theorem \ref{t22}, we need to know  the relation between the ordinary Castelnuovo-Mumford regularity of a finitely generated graded $R$-module and another invariant which we call regularity with respect to the maximal ideal, $\Reg _{\fm}(M)=\max\{\End(H^i_{\fm}(M))+i\}$. In Proposition \ref{pregm} we show that $\Reg(M)\leq \Reg _{\fm}(M)\leq \Reg(M)+\Dim(R_0)$ which generalizes previous results of Hyry \cite{Hy} and Trung \cite{T}. This proposition enables us to state the above formula for the regularity of residual intersection without any restriction on the dimension of $R_0$.

In the presence of the $G_s$ condition, in Lemma \ref{Lulrich}, we prove a graded version of the crucial lemma of Artin and Nagata. With  the aid of this lemma, under the condition $G_s$, if $R$ is Gorenstein and $R_0$ is an Artinian local ring with infinite residue field, the graded structure  of the canonical module of residual intersection is determined  in Proposition \ref{PA-N}, due to the work of Huneke and Ulrich \cite[2.3]{HU}. In this situation, the upper bound obtained for the regularity of  residual intersection is given by  the same formula as in the general case and moreover we can give a criteria to decide when the regularity actually attains the proposed upper bound  in Theorem \ref{t22}.

Finally, in the last section, we consider the residual intersection of perfect ideals of height 2. As it is known,  the Eagon-Northcott complex  provides a free resolution for the residual intersection in this case. Using this resolution, in Theorem \ref{T31}, we exactly determine the Castelnuovo-Mumford regularity of residual intersection of perfect ideals of height 2, whenever the base ring $R_0$ is a field. It is shown that the above formula for the regularity is in fact an equality in this case.

Some of the straightforward verifications which are omitted in the proofs can be found in the Ph.D. thesis of the second author \cite{HA}.


\section{Residual Intersection Without The $G_s$ Condition }

Throughout this section $R$ is a Noetherian ring (of dimension
$d$), $I=(f_1,\cdots,f_r)$ is an ideal of grade $g\geq1$,
$\fa=(l_1,\cdots,l_s)$ is an ideal contained in $I$, $s\geq g$, $
J=\fa:_{R}I$, and $S=R[T_1,\cdots,T_r]$ is a polynomial extension
of $R$ with indeterminates $T_i$'s. We denote the symmetric
algebra of $I$ over $R$ by $\mathcal{S}_I$ and consider
$\mathcal{S}_I$ as an $S$-algebra via the ring homomorphism
$S\rightarrow \mathcal{S}_I$ sending $T_i$ to  $f_i$ as an
element of  $(\mathcal{S}_I)_1=I$. Let $\{
\gamma_1,\cdots,\gamma_s\}\subseteq S_1$ be linear forms whose
images under the above homomorphism are $l_i \in (\mathcal{S}_I)_1$, $(\gamma)$ be the $S$-ideal they generate and
$\frak{g}=(T_1,\cdots,T_r)$. For a sequence of elements $\frak{x}$
in a commutative ring $A$ and an $A$-module $M$, we denote the
koszul complex by $K_\bullet(\frak{x};M)$, it's cycles
$Z_i(\frak{x};M)$ and homologies by $H_i(\frak{x};M)$. For a graded
module $M$, $\beg(M):=\inf\{i : M_i\neq 0\}$ and
$\End(M):=\sup\{i : M_i\neq 0\}$. Setting $\deg(T_i)=1$ for all
$i$, $S$ is a  standard graded over $S_0=R$.

To set one more convention, when we draw the picture of a double
complex obtained from a tensor product of two finite complexes
(in the sense of \cite[2.7.1]{W}), say $\mathcal{A}\bigotimes
\mathcal{B}$; we always put $\mathcal{A}$ in the vertical
direction and $\mathcal{B}$ in the horizontal one. We also label
the module which is  in the up-right corner by $(0,0)$ and
consider the labels for the rest, as the points in the
third-quadrant.

Now, consider the koszul complex
 \begin{center}$K_\bullet(f;R):
0\rightarrow K_r \xrightarrow{\delta_{r}^{f}} K_{r-1}
\xrightarrow{\delta_{r-1}^f} \cdots \rightarrow K_0\rightarrow 0.$
\end{center} Let $Z_i=Z_i(f;R)$.
The $\mathcal{Z}$-complex of $I$ with coefficients in $R$ is
\begin{center}
$\mathcal{Z}_\bullet=\mathcal{Z}_\bullet(f;R): 0\rightarrow
Z_{r-1}\bigotimes_{R}S(-r+1) \xrightarrow{\delta_{r-1}^T} \cdots
\rightarrow Z_1\bigotimes_{R}S(-1) \xrightarrow{\delta_{1}^T}
Z_0\bigotimes_{R}S\rightarrow 0.$
\end{center}
Recall that $Z_r=0$, $H_0(\mathcal{Z}_\bullet)=\mathcal{S}_I$ and
$H_i(\mathcal{Z}_\bullet)$ is finitely generated
$\mathcal{S}_I$-module for all $i$, \cite[4.3]{HSV}.

In order to make our future structures and computations more
transparent, we need to add some intricacies to the
$\mathcal{Z}$-complex.

For $i \geq r-g+1$, the tail of the koszul complex
$K_\bullet(f;R)$ resolves $Z_i$. Now, we construct our first
double complex $\mathcal{F}$ with $\mathcal{F}_{-i,-j}=
K_{r-j+i}\bigotimes_{R}S(-i-r+g-1)$ for  $0 \leq i \leq g-2
\text{~~and~~ } 0 \leq j\leq g-i-2$. $\mathcal{F}$ is a
truncation of $K_\bullet(f;R) \bigotimes_R K_\bullet(T;S)$:
$(\delta:=-r+g-1)$
$$
 \xymatrix{
         &                           &                                                                               &             &  0 \ar[d]                        \\
         &                           &                                                                               &0\ar[r]      &K_{r}\otimes S(\delta) \ar[d]     \\
         &                           &  0 \ar[d]                                                                     &      &\vdots \ar[d]                     \\
         &  0 \ar[r] \ar[d]          & K_r\otimes S(-r+2)\ar[r]^(.7){\partial_{r}'} \ar^{\delta_{r}^f \otimes Id}[d] &\cdots\ar[r] &K_{r-g+3}\otimes S(\delta) \ar[d] \\
 0 \ar[r]&  K_r\otimes S(-r+1) \ar[r]& K_{r-1}\otimes S(-r+2) \ar[r]                                                 &\cdots\ar[r] &K_{r-g+2}\otimes S(\delta)        \\
}
$$
The complex $\mathcal{F}$ is a double complex of free $S$-modules which maps
vertically onto the tail of  $\mathcal{Z}_\bullet$. So that if we
replace the last $g$ modules  of $\mathcal{Z}_\bullet$ by
$\Tot(\mathcal{F})$, with the composition map
$K_{r-g+2}\bigotimes_{R} S(-r+g-1)\xrightarrow{\delta_{r-g+2}^f
\otimes Id} Z_{r-g+1}\bigotimes_{R}
S(-r+g-1)\xrightarrow{\delta_{r-g+1}^T} Z_{r-g}\bigotimes_{R}
S(-r+g)$, then we have a modified $\mathcal{Z}$-complex, say
$\mathcal{Z}'_\bullet$, which has the same homologies as
$\mathcal{Z}_\bullet$, see \cite{HA}, while its tail consists of  free
$S$-modules. Precisely,
\begin{center}
$\mathcal{Z}'_\bullet:= 0 \rightarrow \mathcal{Z}'_{r-1}
\rightarrow \cdots \rightarrow \mathcal{Z}'_{0}
 \rightarrow 0. $
\end{center}
where
$$
\mathcal{Z}'_{i}= \left\lbrace
           \begin{array}{c l}
               K_{i+1}\bigotimes_{R}(\bigoplus_{t=r-i}^{g-1}S(-r-t))  & \text{if $i\geq r-g+1$,}\\
              \mathcal{Z}_i                 & \text{otherwise}.
           \end{array}
         \right.
$$
 Now consider the double complex $\mathcal{E}:=\mathcal{Z'}_\bullet
\bigotimes_S K_\bullet(\gamma;S)$.
Denote $\mathcal{D}_\bullet:=\Tot(\mathcal{E})$ as the following complex,
\begin{center}
$\mathcal{D}_\bullet: 0\rightarrow D_{r+s-1} \rightarrow \cdots
\rightarrow D_1 \rightarrow D_0\rightarrow 0.$
\end{center}
Then $H_0(\mathcal{D}_\bullet)=\mathcal{S}_I/(\gamma)\mathcal{S}_I$
and for all $0 \leq i \leq r+s-1$, the biggest $i$ such that $S(-i)$ appears in the summands of
$D_i$ is $i$, moreover
$$\beg (D_i)=\left\lbrace
           \begin{array}{c l}
              i            & 0\leq i \leq r-g,\\
           r-g+1           & r-g+1\leq i \leq r-1,\\
           i-g+2           & r \leq i \leq  r+s-1.
           \end{array}
         \right.
$$

At the moment, we want to study the properties of the complex
$\mathcal{D}_\bullet$. We shall sometimes use the following lemma.

\begin{lem}\label{l11}
Let $M$ be an $R$-module. Then
\begin{enumerate}
\item[(i)]$H_{\frak{g}}^{i}(M\bigotimes_{R}S)=0$ for all $i\neq
r$,
\item[(ii)]there exists a functorial isomorphism
$\theta_M:H_{\frak{g}}^{r}(M\bigotimes_{R}S)\rightarrow
M\bigotimes_{R}H_{\frak{g}}^{r}(S).$
\end{enumerate}
\end{lem}
\begin{proof}(see\cite[2.1.11]{G}) The proof goes along the same line as in the case
$M=R$. (i) follows from the fact that $T_1, \cdots, T_r$ is a
regular sequence on $M\bigotimes_{R}S$ and (ii) from the
computation of $H_{\frak{g}}^{r}(-)$ via the \v{C}ech complex on $T_1,
\cdots, T_r $.
\end{proof}
The above lemma implies that $H_{\frak{g}}^{j}(D_i)=0$ if $j\neq r$ and $\End(H_{\frak{g}}^{r}(D_i))\leq
-r+i$ for all $i$. In particular,  $H_{\frak{g}}^{r}(D_i)_0=0$ for
all $i\leq r-1$. In the spirit of  \cite[3.2(iv)]{CU} we
introduce the complex $\mathcal{Z}_\bullet^{+}$ of $R$-modules,
\begin{center}
$\mathcal{Z}_\bullet^{+}:=H_{\frak{g}}^{r}(\mathcal{D}_\bullet)_0:
0\rightarrow Z_{r-1}^+ \rightarrow \cdots \rightarrow
Z_{r-s+1}^+\xrightarrow{\varphi_0} Z_{r-s}^+\rightarrow 0.$
\end{center}
Notice that $\mathcal{Z}_i^{+}=Z_{r-s+i}^{+}$. By  Lemma \ref{l11}, for $j\geq r-g+1, Z_j^+$ is a free $R$-module
and for $j\leq r-g, Z_j^+$ is a direct sum of finitely many
copies of some elements of the set
$\{Z_{max\{j,0\}},\cdots,Z_{r-1}\}$.

M.Chardin and B.Ulrich \cite[3.2]{CU} show that under some
conditions on $I$ and $\fa$ the only non-zero homology of this
complex is $\coker \varphi_0 \cong \fa:I$. Our aim in this
section is to extend their result by removing almost all of the
conditions imposed on $I$ and $\fa$ to obtain a sufficient
condition for the acyclicity of $\mathcal{Z}_\bullet^{+}$ and to
determine the structure of $\coker \varphi_0$.
Achieving this aim  sheds some light on the structure of residual
intersections.
 The next lemma is a key to our aim.

\begin{lem}\label{l12}
If $I=\fa$, the only non-zero homology of  $\mathcal{Z}_\bullet^+$
is $\coker \varphi_0 \cong R$.
\begin{proof} Let $\mathcal{C}_{\frak{g}}^{\bullet}(S)$
 be the \v{C}ech complex associated to $\frak{g}$ and $S$. Consider
the double complex
$\mathcal{G}:=\mathcal{C}_{\frak{g}}^{\bullet}(S)\bigotimes_S
\mathcal{D}_\bullet$. By Lemma \ref{l11}, all of the vertical
homologies except those in the base row vanish, therefore
\begin{center}
$\sideset{^1}{_{\text\footnotesize{ver}}^{}}{\E}: 0\rightarrow
H_{\frak{g}}^r(D_{r+s-1}) \rightarrow \cdots \rightarrow
H_{\frak{g}}^r(D_{r+1}) \xrightarrow{\varphi}
H_{\frak{g}}^r(D_r)\rightarrow\cdots\rightarrow
H_{\frak{g}}^r(D_0)\rightarrow 0.$
\end{center}
By definition
$(\sideset{^1}{_{\text\footnotesize{ver}}^{}}{\E})_0=\mathcal{Z}_\bullet^+$.

Now, we return to the (third-quadrant) double complex
$\mathcal{E}$ with $\mathcal{D}_\bullet:=\Tot(\mathcal{E})$, in the case where $I=\fa$. The  vertical spectral
sequence arising from $\mathcal{E}$ at point $(-i,-j)$ has as the
first term $H_j(\mathcal{Z}'_\bullet\bigotimes_S
\bigwedge^{i} S(-1)^s) \cong H_j(\mathcal{Z}_\bullet) \bigotimes_S
\bigwedge^{i} S(-1)^s$. As $H_j(\mathcal{Z}_\bullet)$ is an
$\mathcal{S}_I$-module, it then follows that
$H_i(\mathcal{D}_\bullet)$, for all $i$, is annihilated by a
power of $L= \ker(S \rightarrow \mathcal{S}_I)$. Since $I=\fa$,
$\frak{g}=\frak{g}+L=(\gamma)+L$, hence
$H_{\frak{g}}^j(H_i(\mathcal{D}_\bullet))=H_{(\gamma)}^j(H_i(\mathcal{D}_\bullet))$
for all $i$ and $j$. On the other hand, the  horizontal spectral
sequence (arising from $\mathcal{E}$) at the point $(-i,-j)$ has
as the first term $H_i((\gamma);\mathcal{Z}'_j)$ which
is annihilated by $(\gamma)$. Therefore, the convergence of the
horizontal spectral sequence to the homology modules of
$\mathcal{D}_\bullet$,
implies that $H_i(\mathcal{D}_\bullet)$ is annihilated by some
powers of $(\gamma)$, for all $i$. Hence
$H_{\frak{g}}^j(H_i(\mathcal{D}_\bullet))=H_{(\gamma)}^j(H_i(\mathcal{D}_\bullet))=0$
for all $j\geq1$ and all $i$. Moreover we have,
$\beg(H_{\frak{g}}^0(H_i(\mathcal{D}_\bullet)))\geq \beg(D_i)\geq
1 $ for $i \geq 1$.

Summing up the above paragraph the second horizontal spectral
sequence associated to $\mathcal{G}$ is:
  $$
(\sideset{^{2}}{_{\text\footnotesize{hor}}^{-i,-j}}{\E})_0=H^j_{\frak{g}}(H_i(\mathcal{D}_\bullet))_0=
\left\lbrace
           \begin{array}{c l}
              H^{0}_{(\gamma)}(H_0(\mathcal{D}_\bullet))_0    & \text{if $i=j=0$,}\\
              0                                             & \text{otherwise}.
           \end{array}
         \right.
$$
Now the acyclicity of $\mathcal{Z}_\bullet^+$ and the identification $\coker\varphi_0 \cong H_{\gamma}^0(\mathcal{S}_I/(\gamma)\mathcal{S}_I)_0=(\mathcal{S}_I/(\gamma)\mathcal{S}_I)_0=R$ comes from the fact that $\sideset{^{2}}{_{\text\footnotesize{ver}}^{-i,-j}}{\E} = \sideset{^{\infty}}{_{\text\footnotesize{hor}}^{-i,-j}}{\E}$ for all $i,j$ and the above computation for
$(\sideset{^{2}}{_{\text\footnotesize{hor}}^{-i,-j}}{\E})_0$.

\end{proof}
\end{lem}

The concept of the sliding depth condition $\SD$  first appeared
in the study of the acyclicity of some approximation complexes by Herzog, Simis and Vasconcelos in
\cite{HSV}. This concept was then formally defined by the same
authors in \cite{HSV2}. Let $k$ and $t$ be two integers, we say that the
ideal $I$   satisfies $\SD_k$ at level $t$, if
$\depth(H_i(f;R))\geq \min\{d-g,d-r+i+k\}$ for all $i\geq r-g-t$ (whenever
$t=r-g$ we simply say that $I$ satisfies $\SD_k$, also $\SD$ stands for
 $\SD_0$). However, for our purposes in this section, we
need a slightly weaker condition than the sliding depth condition.
\begin{defn} \label{d13}
Let $k$ and $t$ be two integers. We say that $I$ satisfies the
sliding depth condition on cycles $\SDC_k$ at level $t$, if
$\depth(Z_i)\geq \min\{d-r+i+k, d-g+2, d\}$ for all $r-g-t\leq i\leq r-g$.
\end{defn}

\begin{rem}\label{r14} We make several observations about the
elementary properties of the condition $\SDC$ in the case where $R$ is a CM local ring (see \cite{HA} for some details).
\begin{enumerate}
\item[(i)]The property $\SDC_k$ at
level $t$ localizes and it depends only on $I$, \cite{V}.

\item[(ii)] $\SD_k$  implies $\SDC_{k+1}$, see Proposition \ref{p21} .

\item[(iii)] Whenever $\depth (R)\geq 2$, $\depth (Z_i(f;R))\geq 2$
for all $i$. Furthermore, if $I \neq R$, for all $r-1\geq i\geq r-g+1$, $Z_i$ is a module
of finite projective dimension $r-i-1$. Hence,
$\depth (Z_i)=d-(r-i-1)=d-r+i+1$, for all $r-1\geq i\geq r-g+1$.

\item[(iv)]If $\depth(\Ext _R^i(R/I,R))\geq d-i-1$
for all $i \geq g+1$, for example if $R$ is Gorenstein and $I$ is
CM, then it is not difficult to deduce that $H_{r-g}(f;R)$ is CM
of dimension $d-g$. In this case one can see from  the exact
sequence $0\rightarrow B_{r-g}(f;R)\rightarrow Z_{r-g}\rightarrow
H_{r-g}(f;R)\rightarrow 0$ that $\depth (Z_{r-g})\geq d-g$,
therefore in this case, $I$ satisfies $\SDC_0$ at level $0$.

\item[(v)] In the case where $R$ is Gorenstein local and $I^{unm}$ is CM, where $I^{unm}$ is the unmixed part of $I$, it is shown in Proposition \ref{p15} that $I$ satisfies $\SDC_1$ at level $0$. $\SDC_1$ at level $+1$ is more mysterious, see Example\ref{e18}.

\end{enumerate}
\end{rem}

\begin{prop}\label{p21}
$\SD_k$ implies $\SDC_{k+1}$, whenever $R$ is a CM local ring.
\end{prop}
\begin{proof}
 Consider the truncated Koszul complex
\begin{center}
$0 \ra Z_i \ra K_i \ra K_{i-1} \ra \cdots \ra K_0\ra 0$.
\end{center}
Tensoring the \v{C}ech complex, $\mathcal{C}^\bullet_\fm(R)$, with this
complex, we have the following spectral sequences
$$
\sideset{^{1}}{_{\text\footnotesize{ver}}^{-p,-q}}{\E}=
\left\lbrace
           \begin{array}{c l}
              H^{q}_{\fm}(Z_i)    & \text{ $p=i+1$,}\\
              0 & \text{$p\neq i+1$ and $q \neq d$}, \\
              H^{d}_{\fm}(K_p)    & \text{ $p\neq i+1$ and $q=d$;}
           \end{array}
         \right. $$
so that $\sideset{^{1}}{_{\text\footnotesize{ver}}^{-p,-q}}{\E} \cong \sideset{^{2}}{_{\text\footnotesize{ver}}^{-p,-q}}{\E}$ for all $q \neq d$, and
$\sideset{^{2}}{_{\text\footnotesize{ver}}^{-p,-q}}{\E}=\sideset{^{\infty}}{_{\text\footnotesize{ver}}^{-p,-q}}{\E}$, for any $p$ and $q$.
Recall that $\SDC_{k+1}$ is equivalent to say that$\sideset{^{1}}{_{\text\footnotesize{ver}}^{-p,-q}}{\E}=0$ for $p=i+1$, $i \leq r-g$
 and $q \leq \min\{d-r+k+p-1,d-g+1,d-1\} $.

On the other hand,
$$ \sideset{^{2}}{_{\text\footnotesize{hor}}^{-p,-q}}{\E} =
\left\lbrace
           \begin{array}{c l}
              0  & \text{  $p \geq i$, or  $p \geq r-g-k$ and $q \leq d-g-1$},\\
              0  & \text{  $p \leq \min \{i-1, r-g-k-1\}$ and $q-p \leq d-r+k-1$.}
           \end{array}
         \right. $$
The result, now, follows from the convergence of the spectral
sequences.
\end{proof}

Recall that the unmixed part of an ideal $I$, $I^{unm}$, is the intersection
of all primary components of $I$ with height equal to $\Ht I$. If $I' $ is an
ideal that coincides with $I$ locally in height $\Ht I$ in $\Var (I)$,
 then $I' \subseteq I^{unm}$, \cite[exercise 6.4]{M}. As well, $I^{unm}\subseteq{\rm Ann}(H_{r-g}(f;R)={\rm Ann}(\Ext_R^g(R/I,R)),$ and the equality holds if  $R$ is Gorenstein locally in height $\Ht(I)$.
  Recall that if $R$ is Gorenstein local, then $\omega_{R/I}:=\Ext_R^g(R/I,R)$
 is called the canonical module of $R/I$; in the sense of \cite{HK}.

 In \cite[1.1]{HVV}, Herzog, Vasconcelos and Villarreal  present a lower bound
  for $\depth (Z_{r-g}) $, in the case where $R$ is Gorenstein local and $I$ is CM. In the next proposition
 we  clarify all of the local cohomology modules of $Z_{r-g} $ and exactly
 determine $\depth (Z_{r-g}) $, which gives a complete generalization to \cite[1.1]{HVV}.

\begin{prop}\label{p15}
Suppose that $(R,\fm)$ is Gorenstein and denote by $^\upsilon$ the Matlis dual. Then
\begin{enumerate}
\item[(i)] $H^i_{\fm}(Z_{r-g})\cong H^i_{\fm}(\omega_{R/I})$ for $i<d-g$,

\item[(ii)]  $H^{d-g}_{\fm}(Z_{r-g})\cong (\coker (R/I\xrightarrow{can.} {\rm End}_R (\omega_{R/I})))^\upsilon$,

\item[(iii)] $H^{d-g+1}_{\fm}(Z_{r-g})\cong (I^{unm}/I)^\upsilon$ whenever $g \geq 2$ ,

\item[(iv)] $H_{\fm}^{d-g+i}(Z_{r-g})\cong (H_{i-1}(f;R))^\upsilon$ for $2\leq i \leq g-1$

\item[(v)] $H^{d}_{\fm}(Z_{r-g})\cong I^{unm}$ , if $g=1$.

\end{enumerate}

In particular either $\depth(Z_{r-g}) = \depth(\omega_{R/I})$ or $R/ I^{unm}$ is CM.
In the latter case:

\begin{enumerate}
\item{ $\depth(Z_{r-g}) = d$ if either $g=1$ or $f$ is a regular sequence.}

\item{$\depth(Z_{r-g}) = d-g+1$ if  $I$ is not unmixed.}

 \item{$\depth(Z_{r-g})=d-g+2$ if $g \geq 2$,  $I$ is CM and $f$ is not a regular sequence. }
\end{enumerate}

\end{prop}
\begin{proof}  Consider the short exact sequence $0 \ra B_{r-g} \ra Z_{r-g} \ra \omega_{R/I}(\cong H_{r-g}(f;R)) \ra 0$. Since $B_{r-g}$ is a module of  projective dimension
  $g-1$,  $\depth B_{r-g}=d-g+1 \text{~~and~~}  \Ext_R^i(Z_{r-g},R) \cong \Ext_R^i(\omega_{R/I},R)$ for $i \geq g+1$. Now (i) follows by the local duality.
  Since $\Ext_R^i(\omega_{R/I},R)=0$ for $i \leq g-1$, we have $\Ext_R^{g-i}(Z_{r-g},R)=\Ext_R^{g-i}(B_{r-g},R)$ for all $2 \leq i \leq g$, and the following exact sequence,
  \begin{gather}
  0 \ra \Ext_R^{g-1}(Z_{r-g},R) \ra \Ext_R^{g-1}(B_{r-g},R)\ra \Ext_R^{g}(\omega_{R/I},R) \ra \Ext_R^{g}(Z_{r-g},R) \ra 0. 
  \label{+}
  \end{gather}

To determine all of the $R$-modules, $ \Ext_R^{i}(B_{r-g},R)$, consider the following exact complex, which is a truncation of the Koszul complex $K_\bullet(f;R)$,
$$\mathcal{T}_ \bullet:0\rightarrow K_r\rightarrow \cdots \rightarrow K_{r-g+1} \rightarrow B_{r-g} \rightarrow  0. $$

 Let $\mathcal{I}^\bullet$ be an injective resolution of $R$. The double complex $\Hom_R(\mathcal{T}_{\bullet},\mathcal{I}^{\bullet})$ whose  $(-i)$-th column is $\Hom_R(\mathcal{T}_{r-g+i},\mathcal{I}^j) $ for all $j \geq 0$, gives rise to two spectral sequences where $\sideset{^{1}}{_{\text\footnotesize{hor}}}{\E} = \sideset{^{\infty}}{_{\text\footnotesize{hor}}}{\E} = 0$ and

$$ \sideset{^{2}}{_{\text\footnotesize{ver}}^{-i,-j}}{\E} =
\left\lbrace
           \begin{array}{c l}
             \Ext_R^{j}(B_{r-g},R)  & \text{  $i=0$ and $j \geq 1$ ,}\\
                         0  & \text{  $i \geq 1$ and $ j \geq 1 $,}\\
              H_{g-i}(f;R) & \text{$i \geq 2$ and $j=0$}.
           \end{array}
         \right. $$

Notice that the only non-trivial map arising from this spectral sequence living in $\sideset{^{i}}{_{\text\footnotesize{ver}}}{\E}$ for, $2 \leq i \leq g+1$, is $\sideset{^{i}}{_{\text\footnotesize{ver}}^{0, -i+1}}{\D}:\Ext_R^{i-1}(B_{r-g},R)\ra H_{g-i}(f;R) $.  Therefore, as  $\sideset{^{\infty}}{_{\text\footnotesize{ver}}}{\E}=0$, all of these maps must be isomorphisms, which proves (iv). Also, if $g \geq 2$, it shows that $\Ext_R^{g-1}(B_{r-g},R) \cong R/I.$

We now separate the cases $g=1$ and $ g \geq 2$. First, if $g \geq 2$. Notice that $\Ext_R^{g}(\omega_{R/I},R) \cong {\rm End}_R(\omega_{R/I})$, then, by modifying the maps in \ref{+}, we have the following exact sequence,
 \begin{gather}
  0 \ra \Ext_R^{g-1}(Z_{r-g},R) \ra R/I \xrightarrow{\eta} {\rm End}_R(\omega_{R/I}) \ra \Ext_R^{g}(Z_{r-g},R) \ra 0 , 
  \end{gather}

 where $\eta$ is given by mutiplication by $\eta(1)$. Now, let $\fp \supseteq I$ be a prime ideal of height at most $g+1$, then $\Ext_R^{g}(Z_{r-g},R)_{\fp}=0$, by Remark \ref{r14}[iii], which implies that $\eta(1)$ is unit in ${\rm End}_R(\omega_{R/I})_{\fp}$. The Krull principal ideal theorem applied to the ring  ${\rm End}_R(\omega_{R/I})$ then implies that $\eta(1)$ is unit in ${\rm End}_R(\omega_{R/I})$. Therefore $ \Ext_R^{g}(Z_{r-g},R) \cong \coker \eta \cong \coker(R/I\xrightarrow{can.} {{\rm End}}_R (\omega_{R/I})) $, which yields (ii) for $g \geq 2$.

 For (iii), recall that $\eta$ induces a homomorphism $\bar{\eta}: R/I^{unm}\ra {\rm End}_R(\omega_{R/I}) $  with $\coker \bar{\eta} \cong \coker \eta$. As mentioned above, $\bar{\eta}_\fp$ is onto for all $\fp \subseteq I$ with $\Ht \fp=g$; on the other hand for such a  prime ideal
 $(\omega_{R/I})_{\fp} \cong (R/I)_{\fp} \cong (R/I^{unm})_{\fp},$ hence the composed map from $(R/I)_{\fp}$ to itself is an  isomorphism which implies that $(\bar{\eta})_{\fp}$ is an isomorphism; so that $\bar{\eta}$ is injective, since $\Ass(\Ker \bar{\eta} \subseteq \Ass(R/I^{unm})$. Now, (iii) follows from the commutative diagram below,
 $$
 \xymatrix{
 R/I\ar^(.4){\eta}[r]\ar_{can.}[d] &{\rm End}_R(\omega_{R/I})\\
 R/I^{unm}\ar_{\bar{\eta}}[ur]&\hspace{2cm}.\\
  }
 $$

We now turn to the case $g=1$. To prove (v), note that in this case $B_{r-g}\cong R$, thus the exact sequence \ref{+} can be written as
 \begin{gather}
  0 \ra \Hom_R(Z_{r-1},R) \ra R \xrightarrow{\eta} {\rm End}_R(\omega_{R/I}) \ra \Ext_R^{1}(Z_{r-1},R) \ra 0 
  \label{++}.
  \end{gather}
  One then shows as above that $\eta(1)$ is unit in ${\rm End}_R(\omega_{R/I})$, and  $\Ker \eta= \Ker(R/I\xrightarrow{can.} {\rm End}_R (\omega_{R/I})))={\rm Ann({\rm End}_R (\omega_{R/I}))}= I^{unm}.$

  Replacing $\Ker \eta$ by $I^{unm}$ in \ref{++}, we have the following short exact sequence which immediately completes the proof of (ii) for the case $g=1$,
  $$0\ra R/I^{unm} \xrightarrow{\bar{\eta}} {\rm End}_R(\omega_{R/I}) \ra \Ext_R^{1}(Z_{r-1},R) \ra 0.$$
  By Remark \ref{r14}(iii), $\Ext_R^{1}(Z_{r-1},R)_{\fp}=0$ for all $\fp \supseteq I$ with $\Ht \fp=1,2$. Therefore $\Dim(\Ext_R^{1}(Z_{r-1},R))\leq (\Dim(R/I))-2$. Now, if $R/I^{unm}$ is CM or even satisfies $S_2$, then both $R/I^{unm}$ and ${\rm End}_R(\omega_{R/I})$ satisfy  $S_2$; so that $\depth(\Ext_R^{1}(Z_{r-1},R)_{\fp})\geq 1$, for the same prime ideals $\fp$, which implies that $\Ext_R^{1}(Z_{r-1},R)=0$. Now (1) follows from this fact  and (v), while (2) and (3) are immediate consequences of (i)-(iv). We just mention that, if $\frak{b}$ is an ideal in the Gorenstein ring $R$, then $\omega_{R/\frak{b}}$ is CM and $R/\frak{b}$ is S$_2$ if and only if $R/I^{unm}$ is CM.

  \end{proof}

 We return to the complex $\mathcal{Z}_\bullet^+$ to investigate the acyclicity of this complex. In the next theorem it is shown that the complex  $\mathcal{Z}_\bullet^+$ is acyclic for a wide class of ideals.
\begin{thm}\label{t15}
Suppose that $R$ is a CM local ring and that $J$ is an
$s$-residual intersection of $I$. If $I$ satisfies $\SDC_0$ at
level $\min \{s-g-3,r-g\}$, then $\mathcal{Z}_\bullet^+$ is
acyclic.
\end{thm}
\begin{proof}
Invoking  the "lemme d'acyclicit\'{e}" \cite{PS2} or
\cite[1.4.24]{BH}, we have to show that
\begin{enumerate}
\item[(i)] $\mathcal{Z}_\bullet^+$ is acyclic on the punctured
spectrum, and
\item[(ii)] $\depth (\mathcal{Z}_i^+)\geq i$ for all $i \geq 0$.
\end{enumerate}
(ii) is automatically satisfied due to the condition $\SDC_0$, we
just recall that Remark \ref{r14}(iii) assures that the mentioned level
in the theorem is enough.

To prove (i), let $\fp$ be a non-maximal prime ideal of $R$.
Using induction on $\Ht \fp$, we prove that
$(\mathcal{Z}_\bullet^+)_\fp$ is acyclic. If $\Ht \fp \leq s-1$,
then, by definition of $s$-residual intersection, $\fa R_\fp=
IR_\fp$ which in conjunction with Lemma \ref{l12} implies that
$(\mathcal{Z}_\bullet^+)_\fp$ is acyclic. Now, assume that $\Ht
\fp \geq s$ and that $(\mathcal{Z}_\bullet^+)_\fq$ is acyclic for
any prime ideal $\fq$ with $\Ht \fq < \Ht \fp$. At this moment we
apply the acyclicity's lemma to the complex
$(\mathcal{Z}_\bullet^+)_\fp$. Condition (i)
is satisfied by induction hypothesis. To verify  condition (ii)
for this complex, we consider two cases:
\begin{itemize}
\item{$s-3\leq r$. By Remark \ref{r14}(iii) $\depth ((\mathcal{Z}_i^+)_\fp) \geq
2$ for $i=0,1,2$ (the case where $\depth(R)=1=s$ is trivial). Let  $i\geq
3$, then  keeping in mind the level mentioned in the theorem, we have $\depth ((\mathcal{Z}_i^+)_\fp) = \depth
((Z_{r-s+i}^+)_\fp)= \min\{ \depth ((Z_{r-s+j})_\fp): j \geq i\} \geq \Ht \fp -r
+ (r-s+i) = \Ht \fp -s + i \geq i$.}
\item{$s-3\geq r$. In this case, $(\mathcal{Z}_i^+)_\fp =(Z_0)_\fp\oplus (\oplus_{j\geq 1} (Z_j^{e_{ij}})_\fp) $
for all  $0\leq i \leq s-r$ and some $e_{ij}$.
 Hence, we have to show that $\depth ((Z_i^+)_\fp) \geq s-r +i$ for all $i \geq 0$.
 Remark \ref{r14}(i)
 implies that $\depth ((Z_i^+)_\fp) \geq \Ht \fp -r +i$, and we have  $ \Ht \fp -r +i = \Ht \fp -s +s -r +i \geq s-r+i$ as desired.   }
\end{itemize}
\end{proof}

Now, we identify the module $\coker \varphi_0$.

Consider the two spectral sequences arising from the double
complex $\mathcal{G}$ (see the proof of Lemma \ref{l12}):
\begin{center}
$(\sideset{^{2}}{_{\text\footnotesize{hor}}^{-i,-j}}{\E})_0=H^{j}_{\frak{g}}(H_i(\mathcal{D}_\bullet))_0$
$\hspace{12pt}$for all $i$ and $j$, $(\sideset{^{1}}{_{\text\footnotesize{ver}}^{}}{\E})_0=\mathcal{Z}_\bullet^+$ and\\
$(\sideset{^{2}}{_{\text\footnotesize{ver}}^{-i,-j}}{\E})_0=
\left\lbrace
      \begin{array}{cl}
       H_{i-r}(\mathcal{Z}_\bullet^+) &\text{if}\hspace{3pt} j= r,\\
       0&\text{otherwise.}
       \end{array}
       \right.$

\end{center}
Recall that the degree of a homomorphism in $\sideset{^{i}}{_{\text\footnotesize{hor}}}{\E}$ is $(-i+1,-i)$.
Thus
$(\sideset{^{\infty}}{_{\text\footnotesize{hor}}^{0,0}}{\E})_0\subset (\sideset{^{2}}{_{\text\footnotesize{hor}}^{0,0}}{\E})_0=
H^{0}_{\frak{g}}(H_0(\mathcal{D}_\bullet))_0 \subseteq R $. On the
other hand, by the convergence of
$(\sideset{^{2}}{_{\text\footnotesize{hor}}^{-i,-j}}{\E})_0$ to
the homology modules of
$\mathcal{Z}_\bullet^+$,
there exists a filtration of $H_0(\mathcal{Z}_\bullet^+)=\coker
\varphi_0$, say $\cdots \subseteq \mathcal{F}_2 \subseteq
\mathcal{F}_1 \subseteq \coker \varphi_0$, such that $\coker
\varphi_0/\mathcal{F}_1 \cong
(\sideset{^{\infty}}{_{\text\footnotesize{hor}}^{0,0}}{\E})_0$.
Therefore, defining $\tau$ as the composition of the following
homomorphisms
\begin{gather}
 Z_{r-s}^+ \xrightarrow{can.} \coker \varphi_0 \xrightarrow{can.} \coker
\varphi_0/\mathcal{F}_1 \cong
(\sideset{^{\infty}}{_{\text\footnotesize{hor}}^{0,0}}{\E})_0
\subseteq R,\label{*}
\end{gather}
we have another complex of $R$-modules $\mathcal{C}_\bullet:=
\mathcal{Z}_\bullet^+\xrightarrow{\tau}R \ra 0$.

\begin{prop}\label{p16}
Suppose that $R$ is a CM local ring and that $J$ is an
$s$-residual intersection of $I$. If $I$ satisfies $\SDC_1$ at
level $\min \{s-g-2,r-g\}$, then $\mathcal{C}_\bullet$ is acyclic.
\end{prop}
\begin{proof}
The proof will be in the same way as the proof of Theorem \ref{t15}. Notice that the identification $\coker \varphi_0\cong R$ in Lemma \ref{l12} is given by $\tau$.
\end{proof}

As an application of mentioning the levels in Theorem \ref{t15} and
Proposition \ref{p16}, one can see that in the case where the residual
intersection is close to the linkage the acyclicity of
$\mathcal{Z}_\bullet^+$ and $\mathcal{C}_\bullet$ follows
automatically, without any extra assumption on $I$.
\begin{cor}\label{c17}
If $R$ is a CM local ring and  $J$ is an $s$-residual intersection
of $I$, then
\begin{enumerate}
\item[(a)] $\mathcal{Z}_\bullet^+$ is acyclic if one of the
following conditions holds
  \begin{enumerate}
      \item[(i)]$s \leq g+2$, or
      \item[(ii)] $s=g+3$ and  $H_{r-g}(f;R)$ is CM.
  \end{enumerate}
\item[(b)] $\mathcal{C}_\bullet$ is acyclic if one of the
following conditions holds
  \begin{enumerate}
      \item[(i)]$s \leq g+1$, or
      \item[(ii)] $s=g+2$ ,  $R$ is Gorenstein and $I^{unm}$ is CM.
  \end{enumerate}
\end{enumerate}
\end{cor}
\begin{proof} All parts are immediate consequences of Theorem \ref{t15}, and Proposition \ref{p16}.
Both (a)(i) and (b)(i)  follow  from the fact that $\SDC_1$ at level $-1$
is always satisfied by Remark \ref{r14}(iii). Under the condition of (a)(i),
one can see that $I$ satisfies $\SDC_0$ at level $0$ by Remark \ref{r14}(iv). Also (b)(ii)is implied by  Remark\ref{r14}(v) as $I$ satisfied $\SDC_1$ at level $0$.
\end{proof}

\begin{example}\label{e18}
C. Huneke, in \cite[3.3]{H}, provides an example of a CM ideal $I$
in a regular local ring with a 4-residual intersections which
is not CM. In this example $r=6, s=4$, and $g=3$. Hence
Corollary \ref{c17}(b)(i) shows that the complex $\mathcal{C}_\bullet$,
associated to the ideals in \cite[3.3]{H}, is acyclic. Also, it
will be seen from Theorem \ref{t19} that the ideal $I$  is
an example of a CM ideal in a regular local ring which satisfy
$G_\infty$, generated by a proper sequence \cite[5.5($iv_a$)
and 12.9(2)]{HSV} but doesn't satisfy $\SDC_1$ at level $+1$.
\end{example}

Now, we are ready to establish our main theorem in this section.

\begin{thm}\label{t19}
Suppose that $(R, \fm)$ is a CM *local ring and that $J=\fa:I$  is an
$s$-residual intersection of $I$, with $I$ and $\fa$ are homogeneous ideals. If $I$ satisfies $\SDC_1$ at
level $\min \{s-g,r-g\}$, then either $J=R$, or there exists a homogeneous
ideal $K \subseteq J$ such that
\begin{enumerate}
\item[(i)]$K$ is CM of height $s$;
\item[(ii)]$\Var (K)=\Var (J)$;
\item[(iii)]$K=J$ off $\Var (I)$;
\item[(iv)]$K=J$, whenever $I/\fa$ is generated by at most one element locally in height $s$. In this case $R/J$ is resolved by the complex
$\mathcal{C}_\bullet$ associated to $I$ and $\fa$.
\end{enumerate}
\end{thm}
\begin{proof}  The fact that every ideal we consider is homogeneous enables us to  pass to the local ring $R_{\fm}$. Henceforth we assume $(R,\fm)$ is a CM local ring.

Consider the complex $\mathcal{C}_\bullet$ associated to $I$ and
$\fa$. We prove that the ideal $K=\Image \tau$ satisfies the
desired properties.  The convergence of the spectral sequences
arising from $\mathcal{G}$ in conjunction with
$H_\frak{g}^r(D_r)_1=0$ implies that
$(\sideset{^{\infty}}{_{\text\footnotesize{hor}}^{-i,-i}}{\E})_1=0$
for all $i\geq 0$.  (Further,  one can see that
$(\sideset{^{\infty}}{_{\text\footnotesize{hor}}^{0,0}}{\E})_j=0$
for $j\geq 1$ thus
$(\sideset{^{\infty}}{_{\text\footnotesize{hor}}^{0,0}}{\E})=(\sideset{^{\infty}}{_{\text\footnotesize{hor}}^{0,0}}{\E})_0=\Image
\tau$.) In particular $\frak{g}(\Image \tau)\subseteq
(\sideset{^{\infty}}{_{\text\footnotesize{hor}}^{0,0}}{\E})_1=0$.
That is $\Image \tau \subseteq J$.

If $J=R$, Lemma \ref{l12} implies that $\Image(\tau)=R$, hence to avoid the trivial cases assume, from now on,  that neither $J$ nor $\Image(\tau)$ is the unit ideal.

Notice that by Proposition \ref{p16} the  $\SDC_1$ condition of  $I$   implies
that the complex  $\mathcal{C}_\bullet$ is acyclic.

To prove (i), recall  that for any prime $\fp$ with $\Ht \fp \leq
s-1$, $\fa R_\fp=IR_\fp$, hence by Lemma \ref{l12},
$(\mathcal{C}_\bullet)_\fp \ra 0$ is exact. That is $(\Image
\tau)_\fp=R_\fp$. Thus $\fp$ does not contain $ \Image \tau$.
Therefore $\Ht (\Image \tau)\geq s$.  On the other hand,
considering the double complex
$\mathcal{C}_\fm^\bullet(R)\bigotimes_R\mathcal{C}_\bullet$ the condition $\SDC_1$ on $I$ implies that $\depth(R/\Image \tau)
= \depth(H_0(\mathcal{C}_\bullet))\geq d-s$. Therefore $R/\Image
\tau$ is CM of dimension $d-s$.

For (ii) it is enough to show that $\Var (\Image \tau)\subseteq \Var (J)$.
 Let $\fp$ be a prime ideal that does not contain $J$, then $\fa
R_\fp=IR_\fp$. Then Lemma \ref{l12} implies that $(\Image
\tau)_\fp=R_\fp$ and this completes the proof.

 (iii) is a special case of (iv). For (iv), as $\Image \tau \subseteq J$ and $\Image \tau$
 is CM of height $s$ by (i), it
is enough to show that $\Image \tau$ and $J$ coincide locally in height
$s$. Let $\fp$ be a prime ideal of height $s$. We may (and
do) replace $R$ by $R_\fp$ and assume that $\mu(I/\fa)\leq 1$.
It follows that $\frak{g}\mathcal{S}_I=(\gamma)\mathcal{S}_I + x\mathcal{S}_I$
for some $x \in (\mathcal{S}_I)_1$. Since
$\Supp(H_i(\mathcal{D}_\bullet))\subseteq
\Var((\gamma)\mathcal{S}_I)$ for all $i$ (see Lemma \ref{l12} and it's proof),
$H_\frak{g}^j(H_i(\mathcal{D}_\bullet)) \cong
H_{(x)}^j(H_i(\mathcal{D}_\bullet)) $ for all $j$. Thus
$H_\frak{g}^j(H_i(\mathcal{D}_\bullet))=0$ for all $j\geq 2$. It
then follows that,
$H_\frak{g}^0(\mathcal{S}_I/(\gamma)\mathcal{S}_I)=\sideset{^2}{_{\text\footnotesize{hor}}^{0,0}}{\E}=\sideset{^{\infty}}{_{\text\footnotesize{hor}}^{0,0}}{\E}$.

On the other hand, the following
commutative diagram,
$$\begin{array}{cccccccclcc}
\mathcal{Z}_1^+&\longrightarrow & \mathcal{Z}_0^+ &\longrightarrow &\coker \varphi_0&\longrightarrow &0          &                  && \\
\parallel      &                &\parallel        &                 &\downarrow     &                &\downarrow &                  && \\
\mathcal{Z}_1^+&\longrightarrow & \mathcal{Z}_0^+ &\xrightarrow{\tau}& R            &\longrightarrow&R/\Image\tau & \longrightarrow &0&
\end{array}$$
shows that, the  map $ \coker \varphi_0 \ra \Image\tau$  ~~induced by this diagram is
injective. Then, considering the canonical homomorphisms in (\ref{*}) defining $\tau$,
 $\mathcal{F}_1=0$. This fact implies that,
$(\sideset{^{\infty}}{_{\text\footnotesize{hor}}^{-i,-i}}{\E})_0=0$ for all $i\geq 1$.

Therefore
\begin{center} $\coker \varphi_0 \cong \Image\tau =
\sideset{^{\infty}}{_{\text\footnotesize{hor}}^{0,0}}{\E}=
H_\frak{g}^0(\mathcal{S}_I/(\gamma)\mathcal{S}_I)_0 $.
\end{center} Now, the result follows from the following inclusion
\begin{center}
$\sideset{^{\infty}}{_{\text\footnotesize{hor}}^{0,0}}{\E}=
\Image \tau \subseteq J
 \subseteq H_\frak{g}^0(\mathcal{S}_I/(\gamma)\mathcal{S}_I)_0.$
 \end{center}
\end{proof}

The condition imposed on Theorem \ref{t19}(iv), is not so restricting.
Indeed this condition replace to the conditions $G_s$ and
geometric in other works such as, \cite{C,H,HVV,HU}. Theorem \ref{t19}(iv) is a good progress
to affirmatively answer one of the main open questions in the
theory of  residual intersection \cite[Question5.7]{HU}. As a corollary
one can give a complete answer to this question in the geometric
case.

\begin{cor}\label{c111}
Suppose that $R$ is a CM local ring and $I$ satisfies the sliding
depth condition, $\SD$. Then any geometric residual intersection of $I$
is CM.
\end{cor}

In spite of the complexity of the structure of the ideal $K$ introduced in  Theorem \ref{t19}, it is shown in the next proposition that under some conditions this ideal is a specialization of the generic one.

 We first recall the situations of the generic case. In addition to the notation at the beginning of the section, assume that $(R,\fm)$ is a Noetherian local ring and that $l_i=\sum_{j=1}^rc_{ij}f_j$ for all $i=1,\cdots, s$. Let $U=(U_{ij})$ be a generic $s$ by $r$ matrix, $\tilde{R}=R[U]_{(\fm,U_{ij}-c_{ij})}$, $\tilde{S}=\tilde{R}[T_1,\cdots,T_r]$, $\tilde{l}_i=\sum_{j=1}^r U_{ij}f_j$, $\tilde{\gamma}_i=\sum_{j=1}^r U_{ij}T_j$ for all $i=1,\cdots,s$, $\tilde{\gamma}=(\tilde{\gamma}_1,\cdots,\tilde{\gamma}_s)$, $\tilde{\fa}=(\tilde{l}_1,\cdots,\tilde{l}_s)$ and $\tilde{J}=\tilde{\fa}:_{\tilde{R}}I\tilde{S}$. Consider the standard grading of $\tilde{S}=\tilde{R}[T_1,\cdots, T_r]$ by setting $\Deg(T_i)=1$. Now by replacing the base ring $R$ by the ring $\tilde{R}$, we can construct
the double complex $\tilde{\mathcal{E}}:=\tilde{\mathcal{Z}}'_\bullet \bigotimes_{\tilde{S}} K_\bullet(\tilde{\gamma};\tilde{S})$. Consequently, $\tilde{D}_i=D_i\otimes_S\tilde{S}$. It then follows from the construction of the complex $\mathcal{Z}_{\bullet}^+$ that
\begin{gather*}\label{zz}\tag{2.5}
   \tilde{\mathcal{Z}}_i^+=(H^r_{\frak{g}}(\tilde{S})\otimes_{\tilde{S}}\tilde{D}_i)_0
 \cong( (H^r_{\frak{g}}(S)\otimes_S\tilde{S})\otimes_{\tilde{S}}(D_i\otimes_S\tilde{S} ))_0\\
\cong( (H^r_{\frak{g}}(S)\otimes_S D_i)\otimes_S S[U])_0 \cong
 (H^r_{\frak{g}}(S)\otimes_S D_i)_0\otimes_R R[U]_{(\fm,U_{ij}-c_{ij})}={\mathcal{Z}}_i^+[U]_{(\fm,U_{ij}-c_{ij})}.
\end{gather*}

Before proceeding we recall the definition of  deformation as in \cite[Definition 2.1]{HU1}. Let $(R,\fb)$ and $(\tilde{R},\tilde{\fb})$ be pairs of Noetherian local rings with ideals $\fb \subseteq R$ and $\tilde{\fb}\subseteq \tilde{R}$, we say that  $(\tilde{R},\tilde{\fb})$ is a deformation of $(R, \fb)$ if there exists a sequence $\A \subseteq \tilde{R}$ which is regular on both $\tilde{R}$ and $\tilde{R}/\tilde{\fb}$ such that $\tilde{R}/\A\cong R$ and $(\tilde{\fb}+ \A)/\A\cong \fb$.


\begin{prop}\label{pgeneric}
 With the notation introduced above. If $I$ satisfies SDC$_1$ condition at level $\min\{s-g-2,r-g\}$ and, $J$ and $\tilde{J}$ are $s$-residual intersections of $I$ and $\tilde{I}$, respectively, then $(\tilde{R}, \tilde{K})$ is a deformation of $(R,K)$, via the sequence $(U_{ij}-c_{ij})$.
\end{prop}
\begin{proof} The hypotheses of the proposition in conjunction with Proposition \ref{p16} implies that both $\mathcal{C}_{\bullet}$ and
$\tilde{\mathcal{C}}_{\bullet}$ are acyclic. Moreover, as we mentioned in the proof of Theorem \ref{t19}, in the case where $\mathcal{C}_{\bullet}$ (resp.$\tilde{\mathcal{C}}_{\bullet}$) is acyclic $K\cong \Image(\tau)\cong \coker(\varphi_0)$ (resp. $\tilde{K}\cong \Image(\tilde{\tau})\cong \coker(\tilde{\varphi}_0)$). Let $\pi$ be the epimorphism of $\tilde{R}$ to $R$ sending $U_{ij}$ to $c_{ij}$. One has $\pi(K_{\bullet}(\tilde{\gamma};\tilde{S})) = K_{\bullet}(\gamma;S)$, so that $\pi(\tilde{\mathcal{D}}_{\bullet}) = \mathcal{D}_{\bullet}$, and then (\ref{zz}) shows that $\pi(\tilde{\mathcal{Z}}_{\bullet}^+) = \mathcal{Z}_{\bullet}^+$.
This in turn implies that $\pi(\tilde{K}) = K$.

Clearly, the sequence $(U_{ij}-c_{ij})$  is a regular sequence on $\tilde{R}$. Thus to prove that $(\tilde{R},\tilde{K})$ is a deformation of $(R,K)$ it just remains to prove that  $(U_{ij}-c_{ij})$ is a regular sequence on $\tilde{R}/\tilde{K}$. To this end, consider the double complex $K_{\bullet}(U_{ij}-c_{ij};\tilde{R})\bigotimes_{\tilde{R}}\tilde{\mathcal{C}}_{\bullet}$. In view of (\ref{zz}), $(U_{ij}-c_{ij})$ is a regular sequence on $\tilde{\mathcal{C}}_{i}$ for all $i$. Therefore the first terms in the vertical spectral sequence arising from this double complex has the form $\sideset{^{1}}{_{\text\footnotesize{ver}}^{-p,0}}{\E}\cong \mathcal{C}_{p}$ and $\sideset{^{1}}{_{\text\footnotesize{ver}}^{-p,-q}}{\E}=0$ whenever $q \neq 0$, and $\sideset{^{2}}{_{\text\footnotesize{ver}}^{-p,0}}{\E}\cong H_p(\mathcal{C}_{\bullet})$ .
On the other hand since $\tilde{\mathcal{C}}_{\bullet}$ is acyclic, $\sideset{^{1}}{_{\text\footnotesize{hor}}^{0,-q}}{\E}=K_{q}(U_{ij}-c_{ij};\tilde{R}/\tilde{K})$ and $\sideset{^{1}}{_{\text\footnotesize{hor}}^{-p,-q}}{\E}=0$ if $p\neq 0$,
 and $\sideset{^{2}}{_{\text\footnotesize{hor}}^{0,-q}}{\E}= H_q(U_{ij}-c_{ij};\tilde{R}/\tilde{K})$ , for all $q$.
 Hence both spectral sequences abut at second step and this provides an isomorphism $H_t(\mathcal{C}_{\bullet})\cong H_{t}(U_{ij}-c_{ij};\tilde{R}/\tilde{K})$ for all $t$.
Now the result follows from the acyclicity of $\mathcal{C}_{\bullet}$.

\end{proof}

In the issues concerning the residual intersection, there are some slightly weaker condition than the $G_s$ condition. One of these conditoins which we call $G^-_s$ condition  first appeared in \cite{HSV} to prove the acyclicity of the $\mathcal{Z}$ complex. Similar to the $G_s$ condition, we say that an ideal $I$ satisfies the $G^-_s$ condition if $\mu(I_{\fp})\leq \Ht(\fp)+1$ for all $\fp \supseteq I$ with $\Ht(\fp)\leq s-1$. While the $G_s$ condition is equivalent to existence of  geometric $i$-residual intersections for all $i \leq s-1$, the $G^-_s$ conditions is equivalent to the existence of (not necessarily geometric) $i$-residual intersections for all $i \leq s-1$. The next remark is an extension of \cite[3.8]{CU}.


\begin{rem}\label{cgeneric}With the notation and assumptions as in Proposition \ref{pgeneric}. If in addition $I$ satisfies the $G^-_{s+1}$ condition then $K=\pi(\tilde{J})$. In particular, $K$ only depends on $\fa$ and $I$.

\end{rem}
\begin{proof}
 Once we show that the $G^-_{s+1}$ condition of $I$ implies that $\mu(I\tilde{S}/\tilde{\fa})\leq 1$ locally in height $s$, this remark is an immediate consequence of Theorem \ref{t19}(iv) and Proposition \ref{pgeneric}. We avail ourselves of the proof of \cite[Lemma 3.1]{HU1} to show $\mu(I\tilde{S}/\tilde{\fa})\leq 1$.

Let $Q$ be a prime ideal of Spec$(\tilde{S})$ with $\Ht(Q)\leq s$ and  $\fp=Q \cap \tilde{R}$, let $t:=\Ht(\fp)\leq s$. With the same argument as in proof of  \cite[Lemma 3.1]{HU1} we may assume that $\tilde{I}_{\fp}\tilde{S}_Q$ is generated by at most $t+1$ element in $\tilde{S}_Q$ and assume that $U$ is a $(t+1)\times s$ matrix. Therefore the mapping cone of the following diagram, whose rows are free resolutions, provides a free resolution for $\tilde{I}_{\fp}\tilde{S}_Q/\tilde{\fa}\tilde{S}_Q$.
\[ \xymatrix{
\tilde{S}_Q^m\ar[r]^{\Phi}       & \tilde{S}_Q^{t+1}\ar[r]             &\tilde{I}_{\fp}\tilde{S}_Q\ar[r] &0 &  \\
                                 & \tilde{S}_Q^s    \ar[r]\ar[u]^{U}   & \tilde{\fa}\tilde{S}_Q \ar[u] \ar[r] &0 & \\
}
\]
By the Fitting theorem, to prove the assertion, it is enough to show that $I_t(\Phi\mid U) \not \subseteq Q$
and to this end, it is enough to
show that $I_t (U) \not\subseteq Q$.
If, by contrary, we assume that $I_t(U)\subseteq Q$ then $\Ht(I_t (U)_Q)= (s+1-s+1)(s-t+1)+ \Ht(\fp)= 2s-t+1= s+(s-t)+1 \geq s+1$. Which is a contradiction.
\end{proof}


As it can be seen from the proof of Theorem \ref{t19}(iv), we use the local
condition of generators   on $I/ \fa$ to show that there exists an
element $x$ in $R$ such that $\fa +(x)$ and $I$ have the same radical
 in $\mathcal{S}_I$. So that one may wonder to replace the latter
condition to that in Theorem \ref{t19}(iv). Now, it is natural to ask about the properties of
the ideal $\fa \subseteq I$ in $R$ such that $\fa\mathcal{S}_{I}$ has the same radical as $\mathcal{S}_{I+}$.
 By using the same argument as in the proof of Theorem \ref{t19}(iv), it can be shown that if
 $\fa\mathcal{S}_{I}$ and $\mathcal{S}_{I+}$ have the same radical, then $\fa=I$.

Equivalently, we see in Proposition \ref{P113} that the
 symmetric analogue to the ordinary reduction theory  is vacuous. To be more precise,
 for an ideal $I$ in a commutative ring, we say that the ideal $(\gamma)\subseteq \mathcal{S}_{I} $,
 generated by elements of degree $1$, is a symmetric reduction of $I$, if $Sym^{t+1}(I)=(\gamma)Sym^t(I)$ for some integer $t$.
 Notice that if  $I$ is of linear type, this definition and the known definition of reduction coincide.

Here, we provide an elementary proof to Proposition \ref{P113} which is quite general. Let $A$ be a commutative ring with $1$.
\begin{lem}\label{l112}
Let  $X$ be a set of indeterminates and
$B=A[X]$. If $P$ is an ideal generated by linear forms in $B$
whose radical is $(X)$, then $P= (X)$.
\end{lem}
\begin{proof}

Suppose $x$ is an element of $X$ and $\{p_i\}$ is a set of linear forms generates $P$.
Let $t$ be an integer such that $x^t=\sum_{i=1}^m q_ip_i$ for some integer $m$,  and some $q_i \in B$.
 By homogeneity of $x^t$ and $p_i$, we may assume that each  $q_i$ is homogeneous of degree $t-1$.
 Further if $q_i=b_ix^{t-1}+q'_i$ with $\deg_{x} q'_i < t-1 $, we have $x^t=\sum_{i=1}^m b_ix^{t-1}p_i$. It then follows that
 $x \in P$, since $x^{t-1}$ is a non-zero devisor in $B$.
\end{proof}

\begin{prop}\label{P113}
Let $I$ be an ideal in a commutative ring. Then,
\begin{enumerate}
\item[(i)] $I$ has no proper symmetric reduction, and
\item[(ii)] if $I$ is an ideals of linear type, it has  no (ordinary) proper reduction.
\end{enumerate}
\end{prop}
\begin{proof}
(i) Let $\fa \subseteq I$ be two ideals of a commutative ring $A$ such that $\fa Sym_A(I)$ is a symmetric reduction of $I$.
 $ Sym_A(I)=A[X]/\mathcal{L}$, where $X$ is a set of
indeterminates and  $\mathcal{L}$ is an ideal of linear forms in
$A[X]$. If we denote the preimage of $\fa Sym_A(I)$ in $A[X]$ by $\fa'$, then
the assumptions imply that the radical ideal of
$\fa'+\mathcal{L}$ is $(X)$. Now, the results follows from Lemma \ref{l112}.
(ii)  is an immediate consequence of (i),  since for ideals of linear
type symmetric reductions and ordinary reductions coincide.
\end{proof}

\begin{rem}\label{r114}
 To the best of our knowledge, the fact that ideals of linear type have no proper
 reduction is based on the second analytic deviation and it is
 known in the case where $A$ is a Noetherian local ring. Here, the only assumption is "commutative".

\end{rem}

\newpage
\section{Castelnuovo-Mumford Regularity Of Residual Intersections }
Our  goal in this section is to estimate the regularity of
residual intersections of the ideal $I$, whenever $I$ satisfies some
sliding depth conditions. We use two approaches to this end. One
is based on the resolution of residual intersections which was
already introduced in Theorem \ref{t19}- the complex
$\mathcal{C}_\bullet$; and the second is based on the structure of
the canonical module of the residual intersections- the work of
C.Huneke and B.Ulrich in the local case \cite[2.3]{HU}.
 The benefit of studding the regularity in the second way is that in this way
we obtain a necessary and sufficient condition  for when the regularity gets our proposed upper bound.

Throughout this section $ R=\bigoplus _{n \geq 0} R_{n}$ is a
is a positively graded *local Noetherian ring of dimension $d$
with the graded maximal ideal $\fm$ where the base ring $
R_{0}$ is a  local ring with maximal ideal $ \fm_{0}$. $I$ and
$\fa$ are graded ideals of $R$ generated by homogeneous
elements $f_1,\cdots,f_r$ and $l_1,\cdots,l_s$, respectively,
such that $\deg f_t=i_t$ for all $1 \leq t \leq r$ with $i_1
\geq \cdots \geq i_r$ and $\deg l_{t}=a_{t}$ for $1\leq t \leq
s$. For a graded ideal $\frak{b}$, the sum of the degrees of a
minimal generating set of $\frak{b}$ is denoted by
$\sigma(\frak{b})$.  Keep other notations as in section
2.
 We first recall the definition of the Castelnuovo-Mumford
 regularity.

 \begin{defn}\label{dcmregularity}
 If $M$ is a finitely generated graded $R$-module, the
 Castelnuovo-Mumford regularity  of $M$ is defined as
  $\Reg (M) := \max  \{\END  (H^i_{R_+}(M))+i\}.$

 As an analogue, we define the regularity with respect to the maximal ideal $\fm$, as $\Reg _{\fm}(M) := \max  \{\END
 (H^i_{\fm}(M))+i\}.$
 \end{defn}

In the course of the proof of Theorem \ref{t22} we shall several times use Proposition \ref{pregm}.
This proposition has its own interest as it establishes a relation between $\Reg _{\fm}(M)$ and $\Reg (M)$.
In the proof we shall use the following two elementary  lemmas.


\begin{lem}\label{llc}Suppose that $(A,\fm)$ is a Noetherian local ring  and denote   the Matlis dual $\Hom_A(-,E_A(A/\fm))$ by $^{\upsilon}$. Let $M$ and $N$ be two $A$-module such that $\Dim(N^{\upsilon})>\Dim(M^{\upsilon})$. If $\phi: M\to N$ is an $A$-homomorphism, then $\Dim((\coker \phi)^{\upsilon})=\Dim(N^{\upsilon})$.
\end{lem}

\begin{lem}\label{dimcohd}Suppose that $(A,\fm)$ is a Noetherian complete local ring  and denote   the Matlis dual $\Hom_A(-,E_A(A/\fm))$ by $^{\upsilon}$. Let $M$ be a finitely generated $A$-module. Then $\dim (H^i_\fm (M)^\upsilon)\leq i$ for all $i$, and
equality holds for $i=\dim M$.
\end{lem}


 \begin{prop}\label{pregm} Assume  that $R$ is CM *local and let $M$ be a finitely generated graded $R$-module. Then
 \begin{center}
$ \Reg(M)\leq \Reg_{\fm}(M)\leq \Reg(M)+\Dim(R_0)$.
 \end{center}
 \end{prop}
 \begin{proof}Considering the $\fm_0$-adic completion of $R_0$, $\widehat{R_0}$, we may pass to the CM *complete *local ring $\widehat{R_0}\otimes_{R_0}R$ via the natural homomorphism $R \to \widehat{R_0}\otimes_{R_0}R $; so that in the proof we assume  $R$ admits a canonical module; see \cite[15.2.2]{BS}.

 To prove $\Reg_{\fm}(M)\leq \Reg(M)+\Dim(R_0)$, we consider the composed functor spectral sequence
 $H^p_{\fm_0}(H^q_{R_+}(M))\Rightarrow H^{p+q}_{\fm}(M)$. Let $i$ be an integer.
 Notice that  for all $p > \Dim (R_0)$, $H^p_{\fm_0}(-)=0$, and also if
  $\rho >\Reg(M)+\Dim(R_0)-i$  and $p+q=i$ where $p\leq \Dim (R_0)$,
   then $\rho >\End(H^q_{R_+}(M))+q+\Dim(R_0)-i=\End(H^q_{R_+}(M))+\Dim(R_0)-p\geq
   \End(H^q_{R_+}(M))$. Now, the result follows from the facts that for any integer $q$, $H^q_{R_+}(M)_{\rho}$ is an $R_0$-module, hence $H^p_{\fm_0}(H^q_{R_+}(M)_{\rho})\cong H^p_{\fm_0}(H^q_{R_+}(M))_{\rho}$(c.f. \cite[13.1.10]{BS}); so that we have the following convergence of the components of the above spectral sequence $H^p_{\fm_0}(H^q_{R_+}(M)_{\rho})\Rightarrow
 H^{p+q}_{\fm}(M)_{\rho}$.

 To show that $ \Reg(M)\leq \Reg_{\fm}(M)$, we consider the composed cohomology modules $H^p_{\fm_0}(H^q_{R_+}(M))$ as second terms of the horizontal spectral sequence arising from the double complex $\mathcal{C}_{\fm_0}(R_0)\otimes_{R_0}\mathcal{C}_{R_+}(M)$. As usual, we put this double complex in the third quadrant in the coordinate plane with $\mathcal{C}_{\fm_0}^0(R_0)\otimes_{R_0}\mathcal{C}_{R_+}^r(M)$ at the origin, where $r=\max\{j:H^j_{R_+}(M)\neq 0\}$. So that,  $\sideset{^{2}}{_{\text\footnotesize{hor}}^{-p,-q}}{\E}=H^{q}_{\fm_0}(H^{r-p}_{R_+}(M))$.
 Now, let $i$ be an integer such that $H^i_{R_+}(M)_{\mu}\neq 0$ for $\mu=\Reg(M)-i$.

 Let $\delta=\max\{j:H^{j}_{\fm_0}(H^{t}_{R_+}(M)_{\mu})\neq 0,~~ t \leq i \text{~~and~~} j+t \geq i\}$. Notice that $H^i_{R_+}(M)_{\mu}$ is a finitely generated non-zero $R_0$-module, that is there exists an integer $j$ such that $H^{j}_{\fm_0}(H^{i}_{R_+}(M)_{\mu})\neq 0$, hence $\delta \geq 0$. Let $t\leq i$ be an integer for which $H^{\delta}_{\fm_0}(H^{t}_{R_+}(M)_{\mu})\neq 0$, by definition of $\delta$, $(\sideset{^{2}}{_{\text\footnotesize{hor}}^{-(r-p),-q}}{\E})_{\mu}=0$ for all $p \leq i$ and $q \geq \delta+1$. Thus
 $(\sideset{^{\ell}}{_{\text\footnotesize{hor}}^{-(r-t),-\delta}}{\E})_{\mu}=\coker (\phi_{\ell} )$ with
 $$
 \phi_\ell :=(\sideset{^{\ell}}{_{hor}^{-(r-t)+\ell-1,-\delta+\ell}}{\D})_\mu :
 (\sideset{^{\ell}}{_{\text\footnotesize{hor}}^{-(r-t)+\ell-1,-\delta+\ell}}{\E})_{\mu} \to N_\ell :=(\sideset{^{\ell}}{_{\text\footnotesize{hor}}^{-(r-t),-\delta}}{\E})_{\mu}
 $$

On the one hand, $(\sideset{^{\ell}}{_{\text\footnotesize{hor}}^{-(r-t)+\ell-1,-\delta+\ell}}{\E})_{\mu}^{\upsilon}$ is a subquotient of
the module
$$
(\sideset{^{2}}{_{\text\footnotesize{hor}}^{-(r-t)+\ell-1,-\delta+\ell}}{\E})_{\mu}^{\upsilon}=(H^{\delta -\ell}_{\fm_0} (H^{t-\ell +1}_{R_+}(M)_\mu))^\upsilon
$$
which has dimension at most $\delta -\ell<\delta$ for any $\ell\geq 2$, by Lemma \ref{dimcohd}.

On the other hand, $N_2=H^{\delta}_{\fm_0} (H^{t}_{R_+}(M)_\mu)$, so that $(N_2)^\upsilon$ has dimension $\delta$ by Lemma \ref{dimcohd}.

As $N_{\ell +1}=\coker (\phi_\ell )$, it then follows from Lemma \ref{llc}, by recursion on $\ell$,  that $\Dim((N_{\ell})^{\upsilon})=\delta$ for all $\ell \geq 2$, in particular $(\sideset{^{\infty}}{_{\text\footnotesize{hor}}^{-(r-t),-\delta}}{\E})_{\mu}\neq 0$. Now, the convergence of the spectral sequence, $H^p_{\fm_0}(H^q_{R_+}(M))\Rightarrow H^{p+q}_{\fm}(M)$, implies that $ H^{\delta+t}_{\fm}(M)_{\mu}\neq 0$. Therefore $\Reg_{\fm}(M)\geq \End(H^{\delta+t}_{\fm}(M))+\delta+t\geq \mu+\delta+t\geq \mu +i=\Reg(M)$, as desired.
 \end{proof}


The next proposition follows along the same lines as  the proof of Proposition \ref{pregm}. Since this proposition is not used in the sequel, we will not details the required variations. Part (i) of this proposition was already proved in the articles of E.Hyry \cite{Hy} and N.V.Trung\cite{T}.
\begin{prop}\label{phyry}With the same notations as in  Proposition \ref{pregm}.
\begin{itemize}
\item[(i)]$\max\{\End(H^i_{R_+}(M))\}=\max\{\End(H^i_{\fm}(M))\}$.
\item[(ii)]$H^p_{\fm_0}(H^q_{R_+}(M))=0$ for all integers $p$ and $q$ with $p+q>\Dim(M)$.
\end{itemize}
\end{prop}


We are  ready to present our main result on the regularity of
residual intersections.
\begin{thm}\label{t22}
Suppose that $(R, \fm)$ is CM *local, $I$ is a homogeneous ideal which satisfies
$\SD_1$, $J=\fa:I$ is an $s$-residual intersection of $I$, with $\fa$ homogeneous, and $I/
\fa$ is generated by at most one element locally in height $s$, then
\begin{center}
 $\Reg (R/J) \leq \Reg( R) + \Dim(R_0)+\sigma( \fa)-(s-g+1)\beg(I/\fa)-s $.
\end{center}

\end{thm}
\begin{proof}
 Considering the $\fm_0$-adic completion of $R_0$, $\widehat{R_0}$, the fact that the natural homomorphism $R \to \widehat{R_0}\otimes_{R_0}R $ is  faithfully flat enables us to pass to the CM *complete *local ring $\widehat{R_0}\otimes_{R_0}R$ via this homomorphism; so that in the proof we assume that $R$ admits a graded canonical module.

 The assumptions of the theorem completely fulfill what that is
needed for Theorem \ref{t19}(iv). Thus  $R/J$ is CM and resolved by $\mathcal{C}_{\bullet}$.

Before continuing, we just notice that in  case where $g=1$,
there is no free $R$-module in the tail of
$\mathcal{C}_\bullet$, that is
$\mathcal{C}_s=Z_{r-g}^+=Z_{r-1}^+$. Nevertheless, the coming
proof will be the same for both cases.

 We consider the  diagram of the double complex
$\mathcal{C}_{\fm}^\bullet(R)\bigotimes_R\mathcal{C}_\bullet$,
where  $\mathcal{C}_{\fm}^\bullet(R)$ is the \v{C}ech complex
with respect to $R$ and $\fm$; as usual we put this double complex in the third quadrant with $\mathcal{C}_{\fm}^0(R)\otimes_R \mathcal{C}_0$ at the origin.

By the acyclicity of $\mathcal{C}_\bullet$, we have

$$ \sideset{^{2}}{_{\text\footnotesize{hor}}^{-p,-q}}{\E} =
\left\lbrace
           \begin{array}{c l}
              H^{d-s}_{\fm}(R/J)  & \text{  $p=0$ and $q=d-s$, }\\
              0  & \text{  otherwise.}
           \end{array}
         \right. $$

The fact that $\mathcal{C}_i$ is free for $i\geq s-g+2$ in
conjunction with Proposition \ref{p21} implies that
$ \sideset{^{2}}{_{\text\footnotesize{ver}}^{-p,-q}}{\E} = H^{q}_{\fm}(\mathcal{C}_p)$ is zero if one of the following holds

\begin{itemize}
\item{$p = 0 , q \neq d$.}
\item{$1 \leq p \leq s-g+1$ and $q-p \leq d-s$.}
\item{$p \geq s-g+2$ and $q \neq d$.}
\end{itemize}

It follows that the  only non-zero module $\sideset{^{2}}{_{\text\footnotesize{ver}}^{-p,-q}}{\E}$
with $q-p=d-s$  is $\sideset{^{2}}{_{\text\footnotesize{ver}}^{-s,-d}}{\E}$.
Hence
$H^{d-s}_{\fm}(R/J)=\sideset{^{\infty}}{_{\text\footnotesize{hor}}^{0,-(d-s)}}{\E} =
\sideset{^{\infty}}{_{\text\footnotesize{ver}}^{-s,-d}}{\E}
\subseteq\sideset{^{2}}{_{\text\footnotesize{ver}}^{-s,-d}}{\E} $, and  $\End (H^{d-s}_{\fm}(R/J)) \leq \End (
\sideset{^{2}}{_{\text\footnotesize{ver}}^{-s,-d}}{\E}) $. We now have to estimate $\End
(\sideset{^{2}}{_{\text\footnotesize{ver}}^{-s,-d}}{\E} )$ to
bound the regularity of $R/J$.

In order to estimate $\End
(\sideset{^{2}}{_{\text\footnotesize{ver}}^{-s,-d}}{\E} )$, we need to review the construction of the tail of
$\mathcal{C}_{\bullet}$. The ring $S$, introduced in the first
section, has a  structure as a positively bigraded algebra.
Considering $R$ as a subalgebra of $S$, we write the
degrees of an element $x$ of $R$ as the 2-tuple $(\Deg x,0)$with the second entry
zero. So, let $\deg f_t=(i_t,0)$ for all $1 \leq t \leq r$,
$\deg l_t=(a_t,0)$ for all $1 \leq t \leq s$, $\deg
T_t=(i_t,1)$ for all $1 \leq t \leq s$, and thus $\deg
\gamma_t=(a_t,1)$ for all $1 \leq t \leq s$. With these
notations the $\mathcal{Z}$-complex has the following shape
\begin{center}
$\mathcal{Z}_\bullet: 0\rightarrow
Z_{r-1}\bigotimes_{R}S(0,-r+1) \rightarrow \cdots \rightarrow
Z_1\bigotimes_{R}S(0,-1) \rightarrow
Z_0\bigotimes_{R}S\rightarrow 0.$
\end{center}
Consequently,
 \begin{center}
$\mathcal{Z}'_{r-1}=R(-\sum_{t=1}^{r}
i_t,0)\bigotimes_{R}(\bigoplus_{i=1}^{g-1} S(0,-r+i))$,
\end{center}
and, taking into account that $a_1,\cdots,a_s$ is a minimal generating
set of $\fa$,
\begin{center}
$D_{r+s-1}=R(-\sum_{t=1}^{r}
i_t,0)\bigotimes_{R}(\bigoplus_{i=1}^{g-1} S(-\sigma
(\fa),-s-r+i))=R(-\sum_{t=1}^{r} i_t - \sigma (\fa)
,0)\bigotimes_{R}(\bigoplus_{i=1}^{g-1} S(0,-s-r+i)).$
\end{center}

By definition of $\mathcal{C}_{\bullet}$,
$\mathcal{C}_i=H^r_{\frak{g}}(D_{r+i-1})_{(\ast,0)}$, where by
degree $(\ast,0)$, we mean degree zero in the second entry and
anything in the first entry, in other word degree zero in $S$
with it's natural grading. Hence, as in the proof of
Proposition \ref{pregm} it follows from the composed functor spectral
sequence that
$H^d_{\fm}(H^r_{\frak{g}}(D_{r+i-1})_{(\ast,0)})\cong
H^{r+d}_{\frak{g}+\fm}(D_{r+i-1})_{(\ast,0)}$.

Then,
$\sideset{^{2}}{_{\text\footnotesize{ver}}^{-s,-d}}{\E}=\Ker
(H^{r+d}_{\frak{g}+\fm}(D_{r+s-1})\longrightarrow
H^{r+d}_{\frak{g}+\fm}(D_{r+s-2}))_{(\ast,0)}$. Let $\omega_R$
be the graded canonical module of $R$, then $\omega_S$ exists
and is equal to $\omega_R[T_1,\cdots,T_r](-\sum_{t=1}^r
i_t,-r)$. If $^{\upsilon}$ denotes the Matlis dual,
$\Hom_{R_0}(-,E_0(R/\fm))$, it then  follows from the graded
local duality theorem that
$\sideset{^{2}}{_{\text\footnotesize{ver}}^{-s,-d}}{\E}=
(\coker(\Hom_S(D_{r+s-2},\omega_S)\ra
\Hom_S(D_{r+s-1},\omega_S)))^{\upsilon}_{(\ast,0)}$. Therefore
$\End(\sideset{^{2}}{_{\text\footnotesize{ver}}^{-s,-d}}{\E})=
-\beg(\coker(\Hom_S(D_{r+s-2},\omega_S)\ra\Hom_S(D_{r+s-1},\omega_S))_{(\ast,0)})$.

Now, recall that the map $\theta:D_{r+s-1}\ra D_{r+s-2}$, in
the tail of the complex $\mathcal{D}_{\bullet}$, is defined by
the $2\times1$ matrix $\left(\begin{matrix}
\delta_{s}^{\gamma}\otimes\mathcal{Z'}_{r-1}\\
 \delta'\otimes K_s(\gamma;S)\\
\end{matrix}\right),$
 where
 $\delta_{s}^{\gamma}$, is the last map in the
Koszul complex $K_{\bullet}(\gamma;S)$ and $\delta'$ is the
most-left map in $\mathcal{Z'}_{\bullet}$. So that there exists
an epimorphism from
$\coker(\Hom_S(\delta_{s}^{\gamma},\omega_S))$ to
$\coker(\Hom_S(\theta,\omega_S))$ which yields
that\begin{center}
$-\beg(\coker(\Hom_S(\theta,\omega_S))_{(\ast,0)}) \leq
-\beg(\coker(\Hom_S(\delta_{s}^{\gamma},\omega_S))_{(\ast,0)})$.
\end{center}
Thus to get an upper bound for the regularity, we need to
estimate the latter initial degree. According to the above
mentioned construction of $D_{r+s-1}$ and $\omega_S$, we have


\begin{gather*}
\Hom_S(D_{r+s-1},\omega_S)\\
=\Hom_S(\bigoplus_{i=1}^{g-1}S(-\sum_{t=1}^{r} i_t - \sigma
(\fa),-s-r+i),\omega_R[T_1,\cdots,T_r](-\sum_{t=1}^r i_t,-r))\\
=\Hom_S(\bigoplus_{i=1}^{g-1}S,\omega_R[T_1,\cdots,T_r])(\sigma
(\fa),s-i)\\
  =\bigoplus_{i=1}^{g-1}\Hom_S(S,\omega_R[T_1,\cdots,T_r])(\sigma
(\fa),s-i).
\end{gather*}

Notice that $\Hom_S(\delta_{s}^{\gamma},\omega_S)$ is in fact the
first homomorphism in the Koszul complex
$\Hom_S(K_{\bullet}(\gamma;S),\omega_S)$, therefore\\

\begin{gather*}
\coker(\Hom_S(\delta_{s}^{\gamma},\omega_S))_{(\ast,0)}=
\bigoplus_{i=1}^{g-1}\left(
\frac{\omega_R[T_1,\cdots,T_r]}{(\gamma)\omega_R[T_1,\cdots,T_r]}(\sigma
(\fa),s-i)\right)_{(\ast,0)}\\
=\bigoplus_{i=1}^{g-1}\left(\omega_R(\sigma(\fa),0)\bigotimes_R
\frac{S}{(\gamma)}(0,s-i)\right)_{(\ast,0)}\\
=\bigoplus_{i=1}^{g-1}\left(\omega_R(\sigma(\fa),0)\bigotimes_R
\left(\frac{S}{(\gamma)}\right)_{(\ast,s-i)}\right).
\end{gather*}
At the moment, let $i_n=\beg(I/\fa)$, in this case for all $i <
i_n$, $I_i=\fa_i$ thus $T_1,\cdots,T_{n-1}\in(\gamma)$; so that
$$\left(\frac{S}{(\gamma)}\right)_{(\ast,s-i)}= \bigoplus_{\A _n+\cdots\A
_r=s-i}\left(\frac{(\gamma)+RT_n^{\A _n}\ldots T_r^{\A _r}}{(\gamma)}\right).$$
It then follows that
\begin{gather*}
\beg(\coker(\Hom_S(\delta_{s}^{\gamma},\omega_S))_{(\ast,0)})\\
=\beg\left(\bigoplus_{i=1}^{g-1}\left(\omega_R(\sigma(\fa),0)\bigotimes_R
\left(\frac{S}{(\gamma)}\right)_{(\ast,s-i)}\right)\right)\\
=\min_{i=1}^{g-1}\left\{\beg\left(\omega_R(\sigma(\fa),0)\bigotimes_R
\left(\frac{S}{(\gamma)}\right)_{(\ast,s-i)}\right)\right\}  \\
\geq \beg(\omega_R(\sigma(\fa)))+\min_{i=1}^{g-1}\left\{\beg\left(\bigoplus_{\A
_n+\cdots\A _r=s-i}\left(\frac{(\gamma)+RT_n^{\A _n}\ldots T_r^{\A _r}}{(\gamma)}\right)\right)\right\}\\
\geq \beg(\omega_R(\sigma(\fa)))+(s-g+1)i_n \\
\geq -\Reg(R)+d-\Dim(R_0)-\sigma(\fa)+(s-g+1)\beg(I/\fa),
\end{gather*}
where the last inequality follows from Proposition \ref{pregm}.
 It shows that


 \begin{gather*}
\End(H^{d-s}_{\fm}(R/J))\leq
-\beg(\coker\Hom_S(\delta_{s}^{\gamma},\omega_S))_{(\ast,0)}))\\
\leq \Reg(R)-d+\Dim(R_0)+\sigma(\fa)-(s-g+1)\beg(I/\fa).
\end{gather*}
Again according to Proposition \ref{pregm}, we have $\Reg(R/J) \leq
\End(H^{d-s}_{\fm}(R/J)) +d-s$ which in conjunction with the
above inequality implies that,
$$\Reg(R/J)\leq \Reg(R)+\Dim(R_0)+\sigma(\fa)-(s-g+1)\beg(I/\fa)-s.$$
\end{proof}

 We recall that in the case of linkage, that is when $s=g$, if
 in addition one has $\dim(R_0)=0$, then the inequality in  Theorem \ref{t22} is in fact an equality.
  However when $\Dim(R_0)\neq 0$, the next simple example shows that, in some cases, the
  regularity of residual intersections (or even  linked ideals)
   may be strictly less than the proposed formula.
\begin{example} Let $R_0:=\mathbb{K}[x]_{(x)}$ and $R:=R_0[y]$.
In this case let $I=(y)$, $\fa=(xy)$ and $J=(x)$ be ideals of
$R$. It is now easy to see that $I$ is linked to $J$ by $\fa$.
Therefore the invariants mentioned in Theorem \ref{t22} are
determined as follow, $\Reg(R)=\Reg(R/I)=\Reg(R/J)=0$,
$\Dim(R_0)=1$, $\sigma(\fa)=1$, $s-g+1=1$, $\beg(I/\fa)=1$, and
$\beg(J/\fa)=0$. Therefore the formula is the equality for
$R/J$ and an strict  inequality for $R/I$.
\end{example}

\newpage
\section{Graded Canonical Module of Residual Intersections }


In this section, assume in addition that
$R=\bigoplus_{i=0}^{\infty}R_i$ is a standard positively graded
Noetheian  ring, such that the base ring $(R_0,\fm_0)$ is
Aritinian local with infinite residue field. As well, the ideals $\fa$, $I$, and $J$ assumed to be homogeneous.

Although in our approach to residual intersection we completely remove the $G_s$ condition,
in the presence of the $G_s$ condition, if $R$ is Gorenstein and $R_0$ is an
Artinian local ring with infinite residue field, the graded structure  of the canonical
module of residual intersection can be determined, due to \cite[2.3]{HU}.
Using the structure of the graded canonical module in this situation, we can exactly determine
when the upper bound obtained for the regularity of  residual intersection can be achieved.

 The properties of ideals $\fa$ such that the ideal $\fa:I$ is a residual intersection
 of $I$, have been already studied in some other senses, see for example \cite[4.2]{PU}.
 Here, we concentrate on this point of view to get a homogeneous version of Artin-Nagata's key lemma \cite[lemma 2.3]{AN}.

\begin{defn}\label{l23}
Suppose that $J=\fa:I$ is an (geometric) $s$-residual intersection of
$I$. We say that \textit{$\fa$ has an A-N homogeneous generating
set}, if there exists a homogeneous generating set $
a_1,\cdots,a_s$ of $\fa$ such that $(a_1,\cdots,a_i):I$ is an
(geometric) $i$-residual intersection of $I$ for all $s\geq i\geq
g$.
\end{defn}

As an example, if $J$ is (geometrically) linked to $I$, that is,
when $s=\Ht I$, then the ideal $\fa$ has an A-N homogeneous
generating set. As well, we shall see in Lemma \ref{Lulrich} that if
$I$ is a homogeneous ideal of the CM ring $R$ which  satisfies
$G_\infty$, then, for any residual intersection $\fa:I$ of $I$,
$\fa$ has an A-N homogeneous generating set. Also a generic
residual intersection (if it exists) provides an example of A-N homogeneous
generating set, \cite{HU}. The following lemma is needed in the
proof of Lemma \ref{Lulrich}. We recall that a proof of this lemma in
(non-graded) local case is given in \cite[Theorem 5.8]{M} and \cite[Lemma 1.3]{U}. Also
\cite[2.5]{CEU} detailing what one can imagine about
Lemma \ref{Lulrich}, however, for the proof of \cite[2.5]{CEU} a proof
for Lemma \ref{Lprime} seems indispensable.


\begin{lem}\label{Lprime}
 Let $M=\bigoplus_{j\in {\mathbb{Z}}} M_i$ be a finitely generated graded
 $R$-module minimally generated by homogeneous elements of degrees $d_1 \geq \cdots \geq d_s$ . Then for any finite set of  homogeneous prime ideals
 $\mathcal{P}=\{\frak{p_1},\cdots,\frak{p_n}\}$, there exists a homogeneous element $x \in M$ of degree $d_1$ such that for all $1\leq i
 \leq n$, $\mu((M/(Rx))_{\frak{p_i}})= \max\{0,
 \mu((M)_{\frak{p_i}})-1\}$.
\end{lem}

\begin{proof}We first note that, for a graded $R$-module $L$ and a
homogeneous prime ideal $\fp$, $\mu(L_{(\fp)})=\mu(L_{\fp})$,
where $L_{(\fp)}$ is the homogeneous localization of $L$ at $\fp$.
Therefore in the course of the proof we deal with the homogeneous
localization instead of the usual localization. We may also
assume that $(M)_{\frak{p_i}}\neq 0$ for every $1\leq i
 \leq n$. Let $M^i$ be the preimage of $\frak{p_i}M_{(\frak{p_i})}
\text{~in~} M$. Our aim is to show that  $M_{d_1} \backslash
\bigcup_{i=1}^{n} (M^i)\neq \emptyset$.

Let $m_1,\cdots,m_l$ be a homogeneous minimal generating set of
$M$ with $\Deg m_i= d_i$. If $\fm \in \mathcal{P}$, say $\fm =
\fp_1$, then $m_1 \notin M^1$, since $m_1,\cdots,m_l$ is a
minimal generating set of $M$. For another element $\fp_i \in
\mathcal{P} \backslash \{\fm\}$, as $M \neq M^i $, there exists
$i_j \in \{1,\cdots, s\}$ such that $m_{i_j} \notin M^i$. Hence,
if $c_i= d_{1} - d_{i_j}$, since $(R_0,\fm_0)$ is Artinian and
$R$ is a standard positively graded ring, we have
$R_{c_i}\backslash (\fp_i)_{c_i} \neq \emptyset$. Now, for any
$r_i \in R_{c_i}\backslash (\fp_i)_{c_i} $, $r_im_{i_j} \in
M_{d_1}\backslash M^i$.

Therefore for all $1\leq i \leq n$, $ M_{d_1}\neq M_{d_1}^i$, in
particular by NAK's lemma $ M_{d_1}\neq M_{d_1}^i + \fm_0
M_{d_1}$. Now taking into account that $R_0/\fm_0$ is an infinite
field, we have $ M_{d_1}\neq \bigcup_{i=1}^n (M_{d_1}^i + \fm_0
M_{d_1})$. In particular, $ M_{d_1}\backslash \bigcup_{i=1}^n
M_{d_1}^i \neq \emptyset$, as desired.
\end{proof}

\begin{rem}\label{r25}Keep the same assumptions as in Lemma \ref{Lprime}.
\begin{enumerate}
\item[(i)] If $\fm \notin \mathcal{P}$, then for any $d\geq d_1$
an element of degree $d$ exists such that satisfies the
assertion of the lemma. Indeed in this case, for any $\fp
\in\mathcal{P}$ and $c\geq 0$, $R_c \neq \fp_c$-the fact which is
needed for the proof.
\item[(ii)] If $(R_0,\fm_0)$ is not Artinian, then Lemma \ref{Lprime}
is no longer true. As a counter-example, suppose that
$(R_0,\fm_0)$ is a Noetherian local ring and that $\fp$ and $\fq
$ are two non-maximal prime ideals of $R_0$. Let $X$ be an
indeterminate. Consider $M=R_0/\fp \bigoplus R_0/\fq(-1)$ as a
graded $R=R_0[X]$-module by trivial multiplication. Under these
circumstances for $\mathcal{P}=\{\fp+(X),\fq+(X)\}$, there exists
no appropriate $x$ desired by the lemma.
\end{enumerate}
\end{rem}


\begin{lem}\label{Lulrich}
If $I$ satisfies $G_s$ and $J=\fa:I$ is an $s$-residual
intersection of $I$, then $\fa$ has an A-N homogeneous generating
set.
\end{lem}
\begin{proof}
 Applying Lemma \ref{Lprime} with $\mathcal{P}=\{\fm\}$. The proof is similar to that of \cite[ 1.4]{U},
  we only replace lemma 1.3 in the proof of \cite[Lemma 1.4]{U}, by
  Lemma \ref{Lprime} and note that the set $Q$ employed in \cite[ 1.4]{U}
  entirely consists of homogeneous prime ideals.

\end{proof}


The next lemma is the base in the inductive construction of the
canonical module of residual intersection. A proof of this lemma
can be found in \cite[2.1]{HU} or \cite[2.1]{U}. Here we give a
slightly different proof.

\begin{lem} \label{lw}
Let $R$ be CM and let $\omega_R$ be its canonical module. Let $I$
be a homogeneous ideal of height $g$ such that
$\omega_R/I\omega_R$ is CM and let
$\alpha=\alpha_1,\cdots,\alpha_g$ be a maximal homogeneous
regular sequence in $I$.  Set $J=(\alpha) : I$. Then $R/J$ is CM
and $\omega_{R/J}=\dfrac{I\omega_R}{(\alpha)
\omega_R}(\sigma((\alpha)))$.
\end{lem}
\begin{proof}
Notice that
$({\omega_R}/{(\alpha)\omega_R})(\sigma((\alpha)))=\omega_{R/(\alpha)}$,
\cite[3.6.14]{BH}. So in order to prove the assertion we can pass
to the case where $(\alpha)=0$. By \cite[3.3.10]{BH},
$\Hom({\omega_R}/{I\omega_R},\omega_R)_\fm$ is CM of dimension
$d$. Moreover
$\Hom({\omega_R}/{I\omega_R},\omega_R)=\Hom(R/I,R)=J$,
\cite[13.3.4(ii)]{BS}. Therefore $J$ is CM of dimension $d$; so
that $\depth(R/J)\geq d-1$.

 Now, consider the homogeneous commutative diagram
 $$\begin{array}{cccccccclcc}
0&\longrightarrow &I\omega_R                 &\longrightarrow  &\omega_R                 &\longrightarrow &\omega_R/I\omega_R          &\longrightarrow&0                    &    &  \\
 &                &\downarrow           &                 &\parallel           &                &\downarrow \gamma &               &                     &    & \\
 0&\longrightarrow & \Hom(R/J,\omega_R)       & \longrightarrow &\omega_R                 &\longrightarrow &\Hom(J,\omega_R)        &\longrightarrow&\Ext_{R}^{1}(R/J,\omega_R)&\longrightarrow &0
\end{array}$$
with exact rows. It is straightforward, but messy, to see that
$\gamma$ is the natural homomorphism which is the composition of
the following natural homogeneous homomorphisms:
\begin{multline*}
\omega_R/I\omega_R \longrightarrow \Hom(\Hom(\omega_R/I\omega_R,\omega_R),\omega_R)\\
\simeq \Hom(\Hom(R/I,\Hom(\omega_R,\omega_R)),\omega_R)\\
\simeq \Hom(\Hom(R/I,R),\omega_R) \simeq \Hom(J,\omega_R).
\end{multline*}
Hence by \cite[3.3.10]{BH}, $\gamma_\fm$ is an isomorphism which
implies that $\gamma$ is an isomorphism, thus
$\Ext_{R}^{1}(R/J,\omega_R)=0$, which in conjunction with the
local duality theorem for the local ring $R_\fm$ implies that
$R/J$ is CM of dimension $d$. Also, considering  5-lemma one sees
that $I\omega_R \simeq \Hom(R/J,\omega_R) \simeq \omega_{R/J} $,
\cite[3.6.12]{BH}, which completes the proof.
\end{proof}


\begin{rem}\label{2.2}
Recall that if $I$ is SCM then $\omega_R/I\omega_R$ and all of the
Koszul homology modules of $\omega_R$ with respect to any
generating set of $I$ is CM.(see for example \cite{S}.)
\end{rem}


\begin{prop}\label{PA-N}
Assume $R$ is Gorenstein standard graded with $R_0$ Artinian local and the graded canonical module
$\omega_R=R(b)$, where $b$ is an integer. Suppose that $J = \fa
:I$ is a geometric $s$-residual intersection of $I$. Assume
moreover that $I$ is SCM and satisfies $G_{s}$. Then $R/J$ is CM of
dimension $d-s$, and
$\omega_{R/J}\cong(I+J/J)^{s-g+1}(b+\sigma(\fa))$
\end{prop}
\begin{proof}
The fact that $R/J$ is CM is followed from Theorem \ref{t19}. We continue
to  the proof by using induction on $s-g$. In the case where
$s=g$, $\fa$ can be generated by a homogeneous regular sequence
\cite[1.5.16 and 1.6.19 ]{BH}; and hence the result follows
immediately  from Lemma \ref{lw}.

 Let $s-g>0$. By Lemma \ref{Lulrich} there
exists an A-N homogeneous generating set for $\fa$, say
$\{a_1,\cdots,a_s\}$. (Notice that, under the hypothesis of the
proposition, $\Ht J = s$ by Theorem \ref{t19}, hence $\mu(\fa)=s$. Thus
the length of the A-N homogeneous generating set does not exceed
$s$). Now $J_{s-1}=(a_1,\cdots,a_{s-1}):I$ is a geometric
$(s-1)$-residual intersection of $I$. let $'$ denote the natural
homomorphism from $R$ to $R/J_{s-1}$. By induction hypothesis,
$\omega_{R'}\cong(I')^{s-g}(b+\sigma((a_1,\cdots,a_{s-1})))=(I')^{s-g}(b+\sum_{i=1}^{s-1}\deg
a_i)$. Also by \cite[3.1]{H}, $R'$ is CM, and $I'$ is a height one
SCM ideal in $R'$. Furthermore,  $a'_s$ is a regular element in
$I'$ and $J'=a'_s:I'$, \cite[1.7(d),(f)]{U}. Hence by Lemma \ref{lw},
$\omega_{R/J}\cong I'\omega_{R'}/(a'_s)\omega_{R'}(\deg a_s)\cong
(I')^{s-g+1}/a'_{s}(I')^{s-g}(b+\sum_{i=1}^{s}\deg a_i)$.

Now, applying the same argument as in the last lines of the proof
of \cite[2.3]{HU} for the local ring $R_{\fm}$, it is easy to see
that the natural homogeneous epimorophism from
$(I')^{s-g+1}/a'_{s}(I')^{s-g}$ to $(I+J/J)^{s-g+1}$ is an
isomorphism. Therefore $\omega_{R/J}\cong
(I+J/J)^{s-g+1}(b+\sum_{i=1}^{s}\deg
a_i)=(I+J/J)^{s-g+1}(b+\sigma(\fa))$.
\end{proof}


Now, we are ready to present the second main result of this
section. Here, we give a sharp formula for the
Castelnuovo--Mumford regularity of geometric residual
intersection of SCM ideals which satisfies $G_s$.

\begin{prop}\label{p210}
Assume that $R$ is a Gorenstein standard graded, with $R_0$ Artinian local, and that $I$ is a  SCM
ideal which satisfies $G_s$. If $J$ is a geometric $s$-residual
intersection of $I$, then
\begin{center}
$\Reg (R/J) \leq \Reg( R)+ \sigma( \fa) -(s-g+1)\beg(I/\fa)-s$,
\end{center}
and the equality holds if and only if $((I/\fa)_i)^{s-g+1} \neq 0$, where $i= \beg(I/\fa)$.

\end{prop}
\begin{proof}We first recall that, since $(R_0,\fm_0)$ is an Artinian local
ring, for any $R$-module $L$ and any integer $i$,
$H^i_{\fm}(L)=H^i_{R_+}(L)$. Hence, by the definition of
Castelnuovo--Mumford regularity and Theorem \ref{t19},

\begin{center} $\Reg (R/J) = \MAX \{\END
 (H^i_{R_+}(R/J))+i: 0\leq i \in \mathbb{Z}\}=\END
 (H^{d - s}_{R_+}(R/J))+ d-s.
$
\end{center}

Using the graded local duality theorem
 \cite[13.4.2 and 13.4.5(iv)]{BS}, one has $\END(H^{d -
 s}_{R_+}(R/J))=
 \END(\Hom_{R_0}(\omega_{R/J},E_{R_0}(R_0/\fm_0)))=-\beg(\omega_{R/J})$.
  Therefore, we have to compute $\beg(\omega_{R/J})$. By Proposition \ref{PA-N},
  $\beg(\omega_{(R/J)})=\beg((I+J/J)^{s-g+1}(b+\sigma(\fa)))$. Recalling that under the conditions of the theorem  $I \cap J= \fa$, we have $\beg((I+J/J)^{s-g+1}(b+\sigma(\fa)))=
 \beg((I/\fa)^{s-g+1})-(b+\sigma(\fa)) \geq
 (s-g+1)\beg(I/\fa)-b-
 \sigma(\fa)$. Thus we get $\Reg (R/J) \leq  \sigma( \fa) -(s-g+1)\beg(I/\fa) + b  -s = \Reg( R) + \sigma( \fa) -(s-g+1)\beg(I/\fa)
 -s$. Also the equality holds if and only if $\beg((I/\fa)^{s-g+1})= (s-g+1)\beg(I/\fa)$
\end{proof}

\section{Perfect Ideals Of Height 2}

In this section, using the tool of a generalized koszul complex,
Eagon-Northcott complex, we demonstrate a free resolution for
residual intersections of perfect ideals of height 2. So that, we
show that in this case, we prove the equality of the upper bound which is found in the
previous section.

Throughout, $R$ is a standard graded Cohen--Macaulay ring over a
field $R_0= \mathbb{K}$, $I$ is a homogeneous ideal of $R$ minimally
generated by $f_1,\ldots, f_r \in R $ with $\deg f_j = i_j$ for
$1\leq j \leq r $ and also $i_1\geq \cdots \geq i_{r-u}>
i_{r-u+1}= \cdots = i_n$ for some $1\leq u \leq r $. Let $\fa =
(l_1,\ldots,l_s)$ be a homogeneous $s$-generated ideal of $R$
properly contained in $I$ with $\deg l_i = a_i$ for $1\leq i \leq
s $ and $a_1\geq \cdots \geq a_{k}> a_{k+1}= \cdots =a_s = i_r$.
Let $J= \fa : I$ be an $s$-residual intersection of $I$.

We start with an elementary combinatorial computation, which is
auxiliary in the proof of the main theorem of this section.
 \begin{lem}\label{L:1}
 Let $m$ be a positive integer, then
  \[ \beta_m(t):=(-1)^m\sum_{j=0}^{m}(-1)^jj^t\binom{m}{j} = \left\lbrace
           \begin{array}{c l}
              0 & \text{if $t \leq m-1$},\\
              >0 & \text{otherwise}.
           \end{array}
         \right. \]
 \end{lem}

 \begin{proof}

   Consider the polynomial  $\alpha_m:=(x-1)^m$ and define a sequence of polynomials as
  follow. $\sideset{}{_{m}^{0}}{\A}(x)=\alpha_m(x)$,
  $\sideset{}{_{m}^{1}}{\A}(x)=\sideset{}{_{m}^{'}}{\A}(x)$ and
  $\sideset{}{_{m}^{i+1}}{\A}(x)=(x\sideset{}{_{m}^{i}}{\A}(x))'$
  for $ i\geq 1$ (Here, $'$
  stands for the ordinary derivation). Considering the binomial expansion of $\alpha_m(x)$, it is easy
 to see that $\sideset{}{_{m}^{t}}{\A}(1)=\beta_m(t)$.
Thus we have to show that
  $\sideset{}{_{m}^{0}}{\A}(1)=\cdots=\sideset{}{_{m}^{m-1}}{\A}(1)=0$
  and $\sideset{}{_{m}^{i}}{\A}(1)\geq1$ for $i \geq m$.

 We use induction on $m$. In the case where $m=1$, we have
$\sideset{}{_{m}^{0}}{\A}(x)=(x-1),
\sideset{}{_{m}^{1}}{\A}(x)=1$ for all $i \geq 1$, as desired.

Now, note that $\alpha_m(x)=(x-1)\alpha_{m-1}(x)$ therefore by
a straightforward argument one can show  for $i\geq 1$ that,
$$ \sideset{}{_{m}^{i}}{\A}(x)=\binom{i}{0}\sideset{}{_{m-1}^{0}}{\A}(x)+x\binom{i}{1}\sideset{}{_{m-1}^{1}}{\A}(x)+\cdots+x\binom{i}{i-1}\sideset{}{_{m-1}^{i-1}}{\A}(x)+(x-1)\binom{i}{i}\sideset{}{_{m-1}^{i}}{\A}(x).$$
Now, by induction, the claim follows immediately from the
following equation for $i \geq 1$,
$$\sideset{}{_{m}^{i}}{\A}(1)=\binom{i}{0}\sideset{}{_{m-1}^{0}}{\A}(1)+\binom{i}{1}\sideset{}{_{m-1}^{1}}{\A}(1)+\cdots+\binom{i}{i-1}\sideset{}{_{m-1}^{i-1}}{\A}(1).$$

 \end{proof}

With the  notations mentioned before the above lemma, we have
the main theorem in this section.
\begin{thm}\label{T31}
 If $I$ is a perfect ideal of height
2, then
\begin{enumerate}
 \item[(i)]{$J$ is perfect of height $s$},
 \item[(ii)]{$s-k\leq u$},
 \item[(iii)]{$\Reg (R/J) = \Reg( R) + \sigma(\fa)
     -(s-1)\beg (I) -s$, whenever $s-k\leq u-1$}.
\end{enumerate}
\end{thm}

\begin{proof}
Consider a minimal free resolution for $I$ and $\fa$,
$$\begin{array}{cccccccclcc}
0&\longrightarrow &\bigoplus_{t=1}^{r-1}R(-b_t)&\longrightarrow &\bigoplus_{t=1}^{r}R(-i_t)&\longrightarrow &I        &\longrightarrow&0&    &  \\
 &                &                            &                &\uparrow                  &                &\uparrow &               & &    & \\
 &                &               \cdots       & \longrightarrow&\bigoplus_{t=1}^{s}R(-a_t)&\longrightarrow &\fa      &\longrightarrow&0&    &.
\end{array}$$
The mapping cone of the above diagram provides a free
presentation for $I/\fa$
\[
  \begin{CD}
  \cdots@>>>\bigoplus_{t=1}^{r-1}R(-b_t)\bigoplus_{t=1}^{s}R(-a_t)@>\psi>>\bigoplus_{t=1}^{r}R(-i_t)@>>>I/\fa@>>>0.
  \end{CD}
 \]
By Fitting Theorem \cite[20.7]{Eisenbud} $ \fa : I
=ann(I/\fa)\subseteq \surd(I_r(\psi))$, thus $\grade
I_r(\psi)\geq \grade(\fa : I)=s=(r-1+s)-r+1;$ so that
by\cite[Exercise 20.6]{Eisenbud}
$ I_r(\psi)=\fa :I$ and moreover, the Eagon-Northcott complex of
$\psi$ provides a free resolution for $R/I_r(\psi)=R/J$  say
$$\mathcal{N}_\bullet: 0\rightarrow N_{s-1}[\sigma] \rightarrow \ldots \rightarrow N_{1}[\sigma]\rightarrow  N_{0}[\sigma]\rightarrow R\rightarrow
R/J\rightarrow0.$$with
$N_j=(Sym_j(\bigoplus_{t=1}^{r}R(-i_t))^*\bigotimes\bigwedge^{r+j}(\bigoplus_{t=1}^{r+s-1}R(-c_t))$
 where $c_1,\ldots,c_{r+s-1}$ are
integers such that
$\{c_1,\ldots,c_{r+s-1}\}=\{b_1,\ldots,b_{r-1},a_1,\ldots,a_s\}$
with $c_1\geq\ldots\geq c_{r+s-1}$ and $\sigma=\sigma(\fa)$.
The module $N_j,0\leq j\leq s-1,$ is a graded free module
generated by elements of degrees $-(i_{t_1}+\ldots
+i_{t_j})+(c_{k_1}+\ldots + c_{k_{r+j}})$ with
$t_1\leq\ldots\leq t_j$ and $k_1<\ldots<k_{r+j}$. Observing the
double complex $\mathcal{C}^\bullet_\fm \bigotimes _R
\mathcal{N}_\bullet$, we get two spectral sequences

\[ \sideset{^{\infty}}{_{\text\footnotesize{hor}}^{-i,-j}}{\E} = \sideset{^{2}}{_{\text\footnotesize{hor}}^{-i,-j}}{\E}= \left\lbrace
           \begin{array}{c l}
              H^{d-j}_{\fm}(R/J)    & \text{if $i=0$,}\\
              0 & \text{otherwise}.
           \end{array}
         \right. \]

\[ \sideset{^{1}}{_{\text\footnotesize{ver}}^{-i,-j}}{\E} = \left\lbrace
           \begin{array}{c l}
                            H^d_{\fm}(R)        & \text{if  $i =0$ and $j=d$,}\\
                            H^d_{\fm}(N_{i-1}[\sigma])  & \text{if  $i\geq 1$ and $j=d$,}\\
              0 & \text{otherwise}.
           \end{array}
         \right. \]

 Since $s\leq d$, the correspondence of these two
spectral sequences yields $H^{i}_{\fm} (R/J)=0$ for
$i=0,\ldots,d-s-1.$ As $\Ht J \geq s$, we also have
$H^{i}_{\fm} (R/J)=0$ for $i>d-s$, hence , as $J$ is
non-trivial, we have $\dim(R/J)=\depth(R/J)=d-s$. Therefore $J$
is CM of height $s$. This completes (i).

 To see  (ii), note that $\Ht J = s$ in (i), means that $\fa$ can not be generated
by a less number of generators than $s$. Now, note that as $\mathbb{K}$-vector spaces $\fa_{i_r}$ is a
subspace of $I_{i_r}$, the former is of dimension $s-k$ while
the later is of dimension $u$; so that $s-k \leq u.$
This shows (ii).

 To prove (iii), we first introduce two numerical functions $f$ and $n$.\\
 $f(j):=$ The maximum degree of generators of  $N_j[\sigma]$=$\sum_{t=1}^{r+j}c_t-ji_r-\sigma$, and\\
$n(j):=$ The number of generators of $N_j[\sigma]$ of the
maximum degree $f(j)$.

By Hilbert-Burch theorem, \cite[3.13]{Eisenbud2}, we have
$b_1,\ldots,b_{r-1} > i_r \geq r-1\geq 1$ and
$\sigma=\sum_{t=1}^{r-1}b_t$. On the other hand,
$f(j+1)-f(j)=c_{r+j+1}-i_r$ for $0 \leq j  \leq s-2$ .
Therefore we get the following ordering of $f(j)$'s for  $0
\leq j  \leq s-1$,
   $$ 0< f(0)< \ldots< f(k-1)=f(k)=\ldots=f(s-1).$$

Now, return to the spectral sequence that we mentioned  above.
The spectral sequence
$\sideset{^{1}}{_{\text\footnotesize{ver}}^{-i,-j}}{\E}$ yields
the following exact
 sequence,
 \begin{gather}
  0\rightarrow H^{d-s}_{\fm}(R/J) \rightarrow H^{d}_{\fm}(N_{s-1}[\sigma])\rightarrow \ldots \rightarrow H^{d}_{\fm}(N_0[\sigma]) \rightarrow H^{d}_{\fm}(R)\rightarrow 0. \label{**}
 \end{gather}
 Hence $ \End (H^{d-s}_{\fm}(R/J))\leq
\End(H^{d}_{\fm}(N_{s-1}[\sigma]))$. Equivalently,
\begin{gather*}
\Reg (R/J)-(d-s) \\
\leq \Reg (N_{s-1}[\sigma])-d = \Reg (R) + f(s-1) -d \\
= \Reg (R) + \sum_{i=1}^{s} d_i -(s-1)i_r -d.
\end{gather*}

 To prove the equality we show that $\End (H^{d-s}_{\fm}(R/J))=a+e$ where
$a=f(s-1)$ and $e=\End (H^{d}_{\fm}(R))$. For $0 \leq j  \leq
s-1$, $H^{d}_{\fm}(N_{j}[\sigma])= \cdots\bigoplus
(H^{d}_{\fm}(R)(-a))^{n(j)}$. Hence if  $
H^{d}_{\fm}(R)_e=\mathbb{K} ^t$
 for some $t>0$. Then the $(e+a)$-th strand of \ref{**} is as below
 $$0\rightarrow H^{d-s}_{\fm}(R/J)_{e+a} \rightarrow \mathbb{K} ^{tn(s-1)}\rightarrow \ldots \rightarrow \mathbb{K} ^{tn(k-1)} \rightarrow  0. $$
 Therefore $H^{d-s}_{\fm}(R/J)_{e+a}\neq 0$ if and only if $\sum_{j=k-1}^{s-1}(-1)^jn(j)\neq
 0$. To this end we compute $n(j)$ for $k-1\leq j \leq s-1.$ (
 In the case where $k=0$, $0\leq j \leq s-1$) $n(j)$ consists of two
 parts

 a) Number of choices of $j$ elements from the set
 $ \{ i_{r-u+1},\ldots,i_r \}$ with possible repeated elements;
 that is $~\binom{u+j-1}{u-1}.$

 b)Number of choices of $(j-k+1)$ elements from the set $\{
 a_{k+1},\ldots,a_s\}$without repeated elements; that is  $\binom{s-k}{j-k+1}$.

So, we have  $n(j)=\binom{s-k}{j-k+1}\binom{u+j-1}{u-1}$ for
$k-1\leq j \leq s-1$.  To show that
$\sum_{j=k-1}^{s-1}(-1)^jn(j)\neq 0$, we need to show that
$\sum_{j=k-1}^{s-1}(-1)^j\binom{s-k}{j-k+1}(j+1)\cdots (j+u-1)
\neq 0$. As $(j+1)\cdots (j+u-1)$ is a polynomial of degree
$u-1$ of $j$ with positive coefficients, it is sufficient to
show that $\sum_{j=k-1}^{s-1}\binom{s-k}{j-k+1}j^t$ has the
same sign for each  $0\leq t \leq u-1$ and at least one of them
is non-zero.

 Changing the variable $j$ by $j'=j-k+1$, we
 have to compute
 \begin{center}
 $\sum_{j'=0}^{s-k}(-1)^{j'+k-1}\binom{s-k}{j'}(j'+k-1)^t$.
 \end{center}
  So, it is enough to determine the sign of the summation
 \begin{center}
 $\sum_{j'=0}^{s-k}(-1)^{j'+k-1}\binom{s-k}{j'}j'^t = (-1)^{k-1-(s-k)}\beta_{s-k}(t)=(-1)^{s+1}\beta_{s-k}(t)$, for $0\leq t \leq
 u-1$.
 \end{center}
 As $s-k\leq u-1$,  \ref{L:1} implies that $\beta_{s-k}(t)=0$ for $t \leq
 s-k-1$ and $\beta_{s-k}(t)$ has the same sign as $(-1)^{s+1}$ for
 $s-k \leq t \leq u-1$ which completes the proof.

 \end{proof}

\begin{rem}\label{R:1}
In the case where $s-k=u$, Lemma \ref{L:1} implies that
$\beta_u(t)=0$ for all $t\leq u-1$, that means
$\sum_{j=k-1}^{s-1}(-1)^jn(j)=0$
 thus $\Reg (R/J)<\Reg (R) + \sum_{i=1}^{s}d_i -(s-1)i_r -
 s.$ Indeed in this  case if we have $i_1\geq \cdots
\geq i_{r-v}>
 i_{r-v+1}=\cdots = i_{r-u}> i_{r-u+1}= \cdots = i_n$ and $a_1\geq \cdots
 \geq a_{k-t}> a_{k-t+1}=\cdots=
a_{k}=i_{r-u}> a_{k+1}= \cdots =a_s = i_r$, then by the same
argument as in the proof of Theorem \ref{T31}, one can see that
$t \leq u-v$ and that $\Reg (R/J) = \Reg( R) + \sigma(\fa)
     -(s-1)i_{r-u} -s$, whenever $t < u-v$.

\end{rem}

Continuing in this way by a similar argument as in Theorem
\ref{T31}, one can deduce the next proposition.

\begin{prop}\label{pperfect2} If $I$ is a perfect ideal of
height 2 and $J\neq R$, then $$\Reg (R/J) = \Reg( R) +
\sigma(\fa)  -(s-1)\beg (I/\fa) -s.$$
\end{prop}

\subsection*{Acknowledgment}
This work was done as a part of the author's Ph.D. thesis as a Cotutelle program between Iran and France. I would like to thank my advisor Professor Marc Chardin, for suggesting the problem, for his enormous help and for being so supportive.


\end{document}